\documentclass[12pt]{amsart}
\usepackage{amsmath,amscd,graphics,graphicx,color,a4wide,hyperref,verbatim}
\usepackage{tgpagella}
\usepackage{multicol, fullpage}
\usepackage[usenames,dvipsnames]{xcolor} % allows you to use color names, call this BEFORE you call TikZ
\usepackage{tikz, tikz-3dplot, pgfplots}
\usepackage{tkz-graph}
\usepackage{tikz-cd}
\usetikzlibrary[positioning,patterns] % tikz libraries for relative positioning and silly patterns
\usetikzlibrary{matrix,arrows,decorations.pathmorphing}

\usepackage{textcmds}
\usepackage{epic}
\usepackage{graphicx}
\usepackage{psfrag}
\usepackage{marvosym}
\usepackage{amsmath}
\usepackage{amsfonts}
\usepackage{amssymb}
\usepackage{amsthm}
\usepackage{mathrsfs}
\usepackage{enumitem}
\usepackage[parfill]{parskip} %noindent everywhere

\usepackage{subcaption}
\usepackage{ytableau}

\usepackage{multirow}
\usepackage{tikz}
\usetikzlibrary{tikzmark}
\usepackage{hhline}
\usepackage{dynkin-diagrams}

\usepackage{slashed}

\usepackage{comment}

\usepackage{calligra}
\usepackage{mathrsfs}
\DeclareMathAlphabet{\mathcalligra}{T1}{calligra}{m}{n}
\DeclareFontShape{T1}{calligra}{m}{n}{<->s*[1.5]callig15}{}

\newtheorem{theorem}{Theorem}[section]
\newtheorem{lemma}[theorem]{Lemma}
\newtheorem{lem-def}[theorem]{Lemma-definition}
\newtheorem{proposition}[theorem]{Proposition}
\newtheorem{prop-def}[theorem]{Proposition-definition}

\newtheorem{corollary}[theorem]{Corollary}
\newtheorem{conjecture}[theorem]{Conjecture}

\theoremstyle{definition}
\newtheorem{definition}[theorem]{Definition}

\newtheorem{example}[theorem]{Example}

\newtheorem{remark}[theorem]{Remark}
\numberwithin{equation}{section}

\newtheorem{thm}{Theorem}[section] % the main one
\theoremstyle{plain} % just in case the style had changed
\newcommand{\thistheoremname}{}
\newtheorem{genericthm}[thm]{\thistheoremname}
  
\newtheorem*{genericthm*}{\thistheoremname}
\newenvironment{namedthm*}[1]
  {\renewcommand{\thistheoremname}{#1}%
   \begin{genericthm*}}
  {\end{genericthm*}}

\newcommand{\CC} {\mathbb{C}}

\newcommand{\LL} {\mathbb{L}}

\newcommand{\PP} {\mathbb{P}}

\newcommand{\RR} {\mathbb{R}}
\renewcommand{\SS} {\mathbb{S}}

\newcommand{\VV} {\mathbb{V}}

\newcommand {\shA} {\mathcal{A}}
\newcommand {\shB} {\mathcal{B}}
\newcommand {\shC} {\mathcal{C}}

\newcommand {\shE} {\mathcal{E}}
\newcommand {\shF} {\mathcal{F}}

\newcommand {\shO} {\mathcal{O}}

\newcommand {\shU} {\mathcal{U}}

\newcommand {\shX} {\mathcal{X}}

\newcommand{\blX}{\widetilde{X}}

\def\dbcoh{D^b}
\def\wt{\widetilde}

\let\cal\mathcal
\def\Ac{{\cal A}}
\def\Bc{{\cal B}}
\def\Cc{{\cal C}}

\def\Ec{{\cal E}}
\def\Fc{{\cal F}}

\def\Hc{{\cal H}}
\def\Ic{{\cal I}}

\def\Lc{{\cal L}}

\def\Oc{{\cal O}}
\def\Pc{{\cal P}}

\def\Uc{{\cal U}}

\newcommand {\Aut} {\operatorname{Aut}}

\newcommand {\Ext} {\operatorname{Ext}}
\newcommand{\sExt}{\mathscr{E} \kern -3pt xt}

\newcommand {\Hom} {\operatorname{Hom}}
\newcommand {\sHom}{\mathscr{H}\kern-5pt\mathcalligra{om}}

\newcommand {\Pic} {\operatorname{Pic}}

\newcommand {\Sym} {\operatorname{Sym}}

\newcommand {\Tot} {\operatorname{Tot}}

\newcommand {\arw} {\longrightarrow}

\newcommand{\spinor}{\mathbb S}

\DeclareRobustCommand{\Sec}{\ifmmode\mathsection\else\textsection\fi}

\makeatletter
\newcommand\longleftrightarrowfill@{%
  \arrowfill@\leftarrow\relbar\rightarrow}
\makeatother

\title[]{Derived equivalence for the simple flop of type $D_5$}

\author[M. Rampazzo]{Marco Rampazzo}
\address{MR: 
Department of Mathematics and Computer Science \\ 
University of Antwerp \\ 
Middelheimlaan 1 \\ 
2020 Antwerp, Belgium.}
\email{marco.rampazzo@uantwerpen.be, marco.rampazzo.90@gmail.com}

\author[Y. Xie]{Ying Xie}
\address{YX: 
School of Mathematical Sciences \\ 
Shenzhen University \\ 
518061 Shenzhen \\ 
P.R. China.}
\email{xieying@szu.edu.cn}

\pgfplotsset{compat=1.18}

\begin{document}

\begin{abstract}
    We prove that every simple flop of type $D_5$, i.e., resolved by blowups with exceptional divisor isomorphic to a generalized Grassmann bundle with fiber $OG(4, 10)$, induces a derived equivalence. This provides new evidence for the DK conjecture of Bondal--Orlov and Kawamata. The proof is based on a sequence of mutations of exceptional objects: we use the same argument to prove derived equivalence for some pairs of non-birational Calabi--Yau fivefolds in $OG(5, 10)$, related to Manivel's double--spinor Calabi--Yau varieties. We extend the construction to prove the derived equivalence of Calabi--Yau fibrations, which are described as zero loci in some generalized Grassmann bundles.
\end{abstract}

\maketitle
\section{Introduction}
% The derived category of coherent sheaves is an ubiquitous notion in modern algebraic geometry: due to its interplay with mirror symmetry, mathematical physics, the theory of moduli spaces and many other fields, it has quickly become a popular subject of study, and the main ingredient in the formulation of several longstanding conjectures. However,
The relation of derived categories with birational geometry is still considered to be somewhat mysterious: while the famous reconstruction theorem by Bondal and Orlov depicts a clear picture for Fano and general type varieties \cite{bondalorlovreconstruction}, $K$-trivial manifolds appear to behave in a totally different manner. Derived Torelli-type statements are well known for K3 surfaces \cite{orlov_derived_torelli}, and recently for some hyperk\"ahler varieties \cite{kapustkas_derived_equivalent_hyperkahlers}, but examples of pairs of Calabi--Yau varieties of dimension higher than two which are derived equivalent but not birational have been found \cite{borisovcaldararu, ottemrennemo, borisovcaldararuperry, kuznetsov_G2}.
\subsection*{Birational maps preserving the derived category} While the derived category is far from being a birational invariant in general, following the works of Bondal--Orlov and Kawamata, it is widely expected that a certain class of birational transformations called \emph{K-equivalences} induce a derived equivalence. This is the content of the following conjecture:
\begin{conjecture}[DK conjecture \cite{bondalorlov, kawamata}]\label{conj:DK_conjecture}
    Consider a birational map $\mu:X_1\arw X_2$ of smooth varieites, which is resolved by morphisms $f_i:X_0\arw X_i$ such that $f_1^*K_{X_1}$ and $f_2^*K_{X_2}$ are linearly equivalent (i.e. a K-equivalence). Then $X_1$ and $X_2$ are derived equivalent.
\end{conjecture}
While the typical examples of K-equivalences are standard flops and Mukai flops, many more examples can be constructed from the geometry of homogeneous varieties, the so-called \emph{generalized Grassmann flops} \cite{leung_xie_new}. In Section \ref{sec:D5_flop} we provide a new class of generalized Grassmann flops which satisfy Conjecture \ref{conj:DK_conjecture}.
\subsection*{Double mirror Calabi--Yau manifolds}
The framework of homological mirror symmetry is based on a famous conjecture, which relates the derived category of coherent sheaves of a given variety $X_1$ to a suitably defined category generated by vanishing cycles on a second variety $X_2$, the \emph{mirror} of $X_1$. While for Fano varieties, the mirror is understood in terms of a fibration (the preimage of a superpotential), it is expected that in the Calabi--Yau case, following the ``classical'' mirror symmetry statements, the mirror should be a Calabi--Yau variety even in the homological/derived setting. This picture becomes more complicated if we consider the case of two non-birational Calabi--Yau varieties in the same family, which are derived equivalent: it is expected the existence of two local mirror maps in the neighborhood of different maximally unipotent monodromy points to two different large complex structure points in the mirror family. For this motivation, pairs of non birational, derived equivalent Calabi--Yau varieties are called \emph{double mirrors}. These pairs are rare, and only a handful of constructions are known. In Section \ref{sec:calabi_yau_5folds}, we provide a new example of double mirror pair of Calabi--Yau fivefolds. These varieties are elements of a divisor in a family of intersections of orthogonal Grassmannians studied by Manivel \cite{manivel}.\\
\\
The geometric constructions underlying the flops in Section  \ref{sec:D5_flop} and the pairs of double mirrors in section \ref{sec:calabi_yau_5folds} are remarkably similar. Fix a vector space $V_{10}\simeq \CC^{10}$ and a nondegenerate quadratic form $q$ on $V_{10}$. Consider the orthogonal Grassmannian $OG(4, V_{10})$ of $q$-isotropic $4$-spaces in $V_{10}$: this variety admits two projective bundle structures over the two connected components of $OG(5, V_{10})$ (see Section \ref{sec:orthogonal_grassmannians_and_roofs}). The class of generalized flops we address is given by any flop $\mu:X_+\arw X_-$ resolved by two blowups such that the exceptional divisor of both blowups is isomorphic to $OG(4, V_{10})$ (see Section \ref{sec:simple_flop_type_D5} for the detailed construction). Then, we generalize to a ``relative'' construction over a smooth projective base $B$: Let $G$ be the simple group of Dynkin type $D_5$, and call $P_{i_1,\dots, i_k}$ the parabolic subgroup associated to the choice of the simple roots  $\{i_1,\dots, i_k\}$. The choice of a principal $G$-bundle $\Pc \arw B$ on a smooth projective base $B$ defines a locally trivial fibration $\Pc\times^G G/P_{4,5}\arw B$, with fibers isomorphic to $OG(4, V_{10})$. Our result extends to all flops resolved by two blowups with exceptional divisor isomorphic to $\Pc\times^G G/P_{4,5}\arw B$ (the details of this construction are given in Section \ref{sec:relative_flops}). We prove the following:
%Then, we generalize to the ``relative'' construction, i.e. the case of flops such that the exceptional divisor is an orthogonal Grassmann bundle $\Oc\Gc(4, \Ec)$ where $\Ec$ is a rank $10$ vector bundle over a smooth projective base $B$ (we assume that the fibration $\Oc\Gc(4, \Ec)\arw B$ is locally trivial). We prove the following:
\begin{theorem}\label{main_theorem_flops}
    Let $\mu:X_+\arw X_-$ be a flop as above. Then, there is an equivalence of categories $$D(X_+)\cong D(X_-).$$
\end{theorem}
For clarity of exposition, the proof will be first presented for the case $B = \{pt\}$ and then extended to the general case.\\
\\
Let us now turn our attention to double mirrors. Note that a $(1,1)$ - section of $OG(4, V_{10})$ admits two extremal contractions, again over the two connected components of $OG(5, V_{10})$, with general fibers isomorphic to $\PP^3$. For each of the two contractions, the fiber jumps to a $\PP^4$ exactly over a Calabi--Yau fivefold, and thus, the choice of the $(1,1)$ - section, under the assumption of generality, defines a pair of fivefolds. In Section \ref{sec:calabi_yau_5folds} we will give the details of this constriction, and prove that a general pair of such fivefolds is not birational but derived equivalent. As for the flops, we generalize this picture by replacing the orthogonal Grassmannian $OG(4, V_{10})$ with an orthogonal Grassmann bundle $\Pc\times^G G/P_{4,5}$, describing pairs of derived equivalent Calabi--Yau fibrations over $B$. The notation of the second point of the following theorem is clarified in Section \ref{sec:relative_flops}.
\begin{theorem}\label{main_theorem_pairs} The following statements are true:
\begin{enumerate}
    \item Consider the projections $p_\pm:OG(4, V_{10})\arw OG(5, V_{10})^\pm$, and a general section $s\in H^0(OG(4, V_{10}), \Oc(1,1))$. Then, the two zero loci $Y_\pm \subset OG(5, V_{10})^\pm$ of $p_{\pm*}s$ are derived equivalent, non birational Calabi--Yau fivefolds.
    \item Fix a line bundle $\Lc$ on $\Pc\times^G G/P_{4,5}$ such that the restriction to every fiber over $B$ is surjective on global sections. Then, a general section of $\Lc$ defines a pair of derived equivalent fibrations over $B$, such that the general fiber of each is a smooth Calabi--Yau fivefold.
    
    %\item \textcolor{purple}{use the notation of page 19} Consider a vector bundle $\Ec\arw B$ with $B$ smooth projective, and the projections $p_\pm:\Oc\Gc(4, \Ec)\arw \Oc\Gc(5, \Ec)^\pm$. Call $\Lc$ a line bundle on $\Oc\Gc(4, \Ec)$ which restricts to $\Oc(1,1)$ on every fiber, and such that each of these restrictions is surjective on global sections. Then, given $s\in H^0(\Oc\Gc(4, \Ec), \Lc)$ general, the zero loci $Y_\pm \subset \Oc\Gc(5, \Ec)^\pm$ of $p_{\pm*}s$ are derived equivalent fibrations over $B$ such that the general fiber is a smooth Calabi--Yau fivefold.
    %\textcolor{red}{To define $\Oc\Gc(5, \Ec)$, a family of quadratic forms is needed. That means a morphism of bundles on the base $B: \Ec \longrightarrow \Ec^{\vee}\otimes L.$ When $L$ is not trivial, the SOD of $\Oc\Gc(5, \Ec)$ is not given by the usual Lefschetz decomposition of $OG(5, V_{10})$ since the quadratic form is not non-degenerate everywhere. The degeneracy loci are just the degeneracy loci of the bundle morphism $\Ec \longrightarrow \Ec^{\vee}\otimes L$. The necessary conditions for empty degeneracy loci are both $L$ and $det(\Ec)$ are trivial.} \textcolor{purple}{I am not sure i understand this point. Anyway, i guess we can get rid of the degeneracy loci by playing around with the rank of $\Ec$ and the dimension of $B$ (we choose the former very big and the latter very small, so to speak.}
\end{enumerate}
\end{theorem}
In the proofs of both theorems, the arguments for derived equivalence follow from a sequence of mutations of exceptional objects. In particular, a consequence of the similarity of the two constructions is the fact that the pattern of mutations is essentially the same. This feature is expected to hold for a wider class of flops/double mirror candidates. In fact, these parallel constructions can be carried out for any homogeneous variety of Picard number two admitting two projective bundle structures \cite{mypaper_roofbundles}. Such homogeneous varieties have been classified by Kanemitsu \cite{kanemitsu}. While several cases have already been settled \cite{segal_C2, kuznetsov_G2, ourpaper_cy3s, ourpaper_k3s, ying_D4_Gdagger}, a general strategy to understand the problem is still missing.

\subsection*{Acknowledgments} The first part of this work took place while the authors were visiting the Chinese University of Hong Kong. The authors also thank Akihiro Kanemitsu, Micha\l\ Kapustka, Naichung Conan Leung, Laurent Manivel, and Maxim Smirnov for enlightening discussions and valuable suggestions. MR was supported by PRIN2020KKWT53, Fonds voor Wetenschappelijk Onderzoek -- Vlaanderen (FWO, Research Foundation -- Flanders) -- G0D9323N, and by the European Research Council (ERC) grant agreement No.~817762.
\section{The flop of type \texorpdfstring{$D_5$}{}}\label{sec:D5_flop}
\subsection{Orthogonal Grassmannians and roofs}\label{sec:orthogonal_grassmannians_and_roofs}
Let us call $G(k, V_n)$ the Grassmannian of $k$-dimensional subspaces of a fixed vector space $V_n\simeq\mathbb C^n$. We call $\mathcal U$ the tautological bundle of $G(k, V_n)$, i.e., the rank $k$ vector bundle whose total space is described by the following incidence correspondence:
\begin{equation*}
    \mathcal U = \{ (x, v)\in G(k, V_n)\times V_n : v\in x\},
\end{equation*}
where we identify points in $G(k, V_n)$ with the corresponding $k$-spaces in $V_n$.\\
The orthogonal Grassmannian $OG(k, V_{2n})$ is the zero locus of a general section of $\Sym^2\mathcal U^\vee$ in $G(k, V_{2n})$ for any $k\leq n$. If we fix $k=n$, such variety has two isomorphic connected components $OG(n, V_{2n})^\pm\subset \mathbb P^{2^{n-1}-1}$, each of them is called a \emph{spinor variety}. As rational homogeneous varieties, spinor varieties are described as $OG(n, V_{2n})^+ = D_n/P_{n-1}$ and $OG(n, V_{2n})^- = D_n/P_{n}$, i.e., they are the quotients of the simple Lie group of Dynkin type $D_n$ (i.e., $\operatorname{Spin}(2n)$) by the parabolic subgroups associated to the last two simple roots in the Bourbaki indexing convention.\\
More geometrically, the varieties $OG(n, V_{2n})^\pm$ are maximal isotropic (orthogonal) Grassmannians, i.e., they describe the two disjoint families of $n-1$-dimensional linear spaces contained in a quadric hypersurface $Q_{2n-2}\in\mathbb P^{2n-1}$. Observe that any $\mathbb P^{n-2}\subset Q_{n-2}$ can be extended to an isotropic $\mathbb P^{n-1}$ in exactly two ways: this gives $OG(n-1, V_{2n})$ two projective bundle structures $p_\pm: OG(n-1, V_{2n})\longrightarrow OG(n, V_{2n})^\pm$, both with the same relative dimension $n-1$ and the same relative ample generator $\mathcal O(1,1):=p_+^*\mathcal O(1)\otimes p_-^*\mathcal O(1)$. Homogeneous varieties with these properties are called \emph{roofs}, and they have been classified by Kanemitsu \cite[Example 5.12]{kanemitsu}: in his notation, $OG(n-1, V_{2n})$ is known as the roof of type $D_n$.

\subsection{The roof of type \texorpdfstring{$D_5$}{} and the associated flop}\label{sec:simple_flop_type_D5}
In the following, we focus on the case $n = 5$. Define $\spinor_\pm:=OG(5, V_{10})^\pm$. Let us call $\Uc_{\pm}$ the taulogical bundles of, respectively, $\spinor_{\pm}$, which are just the restriction from $\Uc_+$ on $Gr(5, V_{10})$. The roof structure of $OG(4, V_{10})$ can be summarized by the following diagram:
\begin{equation}
\begin{tikzcd}[row sep = huge]\label{diagram_D5_roof}
                        &  OG(4, V_{10}) \arrow[dl, swap, "p_+"] \arrow[dr,"p_-"] \\
                         \spinor_+ &  &\spinor_-,
\end{tikzcd}
\end{equation}
where both $p_{\pm}$ are the projective bundles given by:
$$ OG(4, V_{10})=\PP(\Uc_+\longrightarrow \spinor_+),$$
$$ OG(4, V_{10})=\PP(\Uc_-\longrightarrow \spinor_-).$$
Let $h_{\pm}$ be the ample generator class of $\Pic(OG_{\pm}(5, V_{10})$. Then:
$$p_{\pm*}(\shO(1, 1))=\shU_\pm(2).$$
The relative Euler sequences of $p_\pm$ can be written as follows (we omit pullbacks):
\begin{equation}\label{relative_euler}
    0\longrightarrow \Uc_4(2, 0) \longrightarrow \shU_+(2, 0) \longrightarrow \shO(1, 1) \longrightarrow 0,
\end{equation}
\begin{equation}\label{relative_euler_minus_side}
    0\longrightarrow \Uc_4(0, 2) \longrightarrow \shU_-(0, 2) \longrightarrow \shO(1, 1) \longrightarrow 0,
\end{equation}
where $\Uc_4$ is the tautological bundle of $OG(4, V_{10})$.
\begin{definition}
    We call \emph{simple flop of type} $D_5$ over a point a birational map $\mu: X_+\longrightarrow X_-$ such that:
    \begin{enumerate}
        \item it is resolved by smooth blowups, and 
        \item the exceptional divisor is isomorphic to $OG(4,10)$.
    \end{enumerate}
\end{definition}
\begin{remark}
    In a $K$-equivalence resolved by smooth blowups, the two exceptional divisors coincide, which is followed immediately by the linear equivalence of the pullbacks of the canonical classes.
\end{remark}
A simple flop of type $D_5$ over a point can be summarized with the diagram:
\begin{equation}
\begin{tikzcd}[row sep = large]\label{diagram_D5_flop_over_a_point}
&&E=OG(4,10) \arrow[ddll,swap,"p_+"] \arrow[d, hook,"j"]\arrow[ddrr,"p_-"] \\
&&  \arrow[dl,swap,"\pi_+"]  \widetilde{X} \arrow[dr,"\pi_-"]\\
  \spinor_+ \arrow[r,hook,"i_+"]& X_+ \arrow[rr,dashed, "\mu"]    & & X_- & \arrow[l, hook',swap, "i_-"]\spinor_-. 
\end{tikzcd}
\end{equation}
We give some examples of this construction. %\textcolor{purple}{Are they all of them? Or can we cook up something more? Is the condition on the normal bundle of $\SS_+\subset X_+$ very restrictive?}\textcolor{red}{This is related to the existence of a holomorphic tubular neighborhood. According to Cor 3.7, in https://arxiv.org/abs/math/0612449, there is a formal tubular neighborhood. The obstruction classes lie in $H^1(\spinor, N\otimes Sym^k N^*)$ and $H^1(\spinor, T\otimes Sym^k N^*)$ for all $k$. $N$ is the normal bundle.}\textcolor{purple}{see also Remark 11.32 in huybrechts}
\begin{example}
    Consider the quasi-projective varieties $X_{+}=\Tot_{\spinor_+}(\shU  ^{\vee}(-2h_{+}))$ and $X_{-}=\Tot_{\spinor_-}(\Uc_4  ^{\vee}(-2h_{-}))$. Then 
    $$\blX=\Tot_{OG(4,10)}(\shO(-1, -1))=Bl_{\spinor_+}X_+=Bl_{OG_-(5,10)}X_-.$$
\end{example}
\begin{example}
    Take the projective bundles $X_{\pm}=\PP(\shU  ^{\pm\vee}(-2h_{\pm})\oplus\Oc\arw OG_{\pm}(5,10))$. Then we can have
    $$\blX= \PP(\shO(-1, -1)\oplus\Oc\arw OG(4,10))=Bl_{\spinor_+}X_+=Bl_{\spinor_-}X_-.$$
\end{example}
\begin{remark}
    Note that the condition on the normal bundle of $\SS_\pm\subset X_\pm$ does not determine $X_\pm$, but only a formal neighborhood of $\SS_\pm$, hence many more examples of $D_5$ flops could exist in nature. Additional details about the relation between $\SS_\pm$ and its ambient variety can be found in \cite[Remark 11.32]{huybrechts_fm_transform} and in \cite{abate_bracci_tovena}.
\end{remark}
\subsection{Cohomology computations} This section is preparatory to the main proof. We address here the cohomology computations needed to perform the mutations in the next section, presenting the main techniques with explicit computations. The principal tool will be reducing the problem to computing cohomology of homogeneous irreducible vector bundles. For this purpose, let us fix a set of fundamental weights $\{\omega_1, \dots, \omega_5\}$ for the simple algebraic group of type $D_5$. Then, homogeneous irreducible bundles on $OG(4, V_{10})$ will be denoted by $\Ec_{a_1\omega_1+\dots+a_5\omega_5}$, where $a_1, a_2, a_3$ are nonnegative, while on $\SS_+$ (resp. $\SS_-)$ the fourth (resp. fifth) coordinate will have to be nonnegative as well. In particular, one has the following identities on $OG(4, V_{10})$:
\begin{equation}\label{weights_U4_shur_powers}
    \begin{split}
        \wedge^k\Uc_4^\vee(a, b) & = \Ec_{\omega_k+a\omega_4+b\omega_5} \hspace{5pt} \text{for} \hspace{5pt} 1\leq k \leq 3,\\
        \wedge^4\Uc_4^\vee(a, b) & = \Ec_{(a+1)\omega_4+(b+1)\omega_5} \hspace{5pt},\\
        \Sym^k\Uc_4^\vee(a, b) & = \Ec_{k\omega_1+a\omega_4+b\omega_5} \hspace{5pt} \text{for} \hspace{5pt} k\geq 0,
    \end{split}
\end{equation}
and the following on $\spinor_+$:
\begin{equation}\label{weights_U_shur_powers}
    \begin{split}
        \wedge^k\Uc_+^\vee(a, 0) & = \Ec_{\omega_k+a\omega_4} \hspace{5pt} \text{for} \hspace{5pt} 1\leq k \leq 3,\\
        \wedge^4\Uc_+^\vee(a, 0) & = \Ec_{(a+1)\omega_4+b\omega_5} \hspace{5pt},\\
        \wedge^5\Uc_+^\vee(a, 0) & = \Ec_{(a+2)\omega_4} \hspace{5pt},\\
        \Sym^k\Uc_+^\vee(a, 0) & = \Ec_{k\omega_1+a\omega_4} \hspace{5pt} \text{for} \hspace{5pt} k\geq 0.
    \end{split}
\end{equation}
Clearly, similar identities hold for $\Uc_-$ on $\spinor_-$ with the roles of $\omega_4$ and $\omega_5$ swapped. Since the Levi subgroup of $OG(4, V_{10})$ is isomorphic to $SL(4)\times\CC^*\times\CC^*$, and the one of $\spinor_\pm$ is isomorphic to $SL(5)\times\CC^*$, tensor products of the bundles above can be themselves decomposed in semisimple factors with the usual Littlewood--Richardson formula. Cohomology is easily computed by the Borel--Weil--Bott theorem in an algorithmic way:
\begin{enumerate}
    \item If $\lambda$ has no negative entries, the bundle is globally generated, and its only nonvanishing cohomology is $H^0(E, \Ec_\lambda) =\VV_\lambda^{D_5}$.
    \item Otherwise, consider the vector $\lambda+\rho$ where $\rho = (1, 1, 1, 1, 1)$ and apply the Weyl reflection corresponding to the leftmost negative entry in the vector. If $(\lambda+\rho)_i = 0$ for some $i$ and $(\lambda+\rho)_j \geq 0$ for all $j\neq i$, then $\Ec_\lambda$ has no cohomology. If not, repeat the last step until either we satisfy this condition (and the bundle has no cohomology) or we obtain a vector $\wt\lambda$ with all strictly positive entries.
    \item If we get a vector $\lambda$ with all strictly positive entries after at least an iteration of (2), the only nonvanishing contribution to cohomology is $H^d(E, \Ec_\lambda) = \VV_{\wt\lambda - \rho}^{D_5}$ where $d$ is the number of Weyl reflections to obtain $\wt\lambda$, i.e, the number of iterations of step (2).
\end{enumerate}
Explicitly, Weyl reflections act on a weight vector as follows:
\begin{equation*}
    \begin{split}
        s_1: (a_1, a_2, a_3, a_4, a_5) &\arw (-a_1, a_1+a_2, a_3, a_4, a_5), \\
        s_2: (a_1, a_2, a_3, a_4, a_5) &\arw  (a_1+a_2, -a_2, a_2+a_3, a_4, a_5), \\
        s_3: (a_1, a_2, a_3, a_4, a_5) &\arw (a_1, a_2+a_3, -a_3, a_3+a_4, a_3+a_5), \\
        s_4: (a_1, a_2, a_3, a_4, a_5) &\arw (a_1, a_2, a_3+a_4, -a_4, a_5), \\
        s_5: (a_1, a_2, a_3, a_4, a_5) &\arw (a_1, a_2, a_3+a_5, a_4, -a_5).
    \end{split}
\end{equation*}
This elementary algorithm can be easily automatized: see, for instance, the script \cite{python_script} for a basic Python implementation.
\begin{lemma}\label{we_can_use_koszul}
    For any sheaves $F, G$, one can resolve $ \Ext_X^m(j_*F, j_*G)$ by means of the following exact sequence:
    \begin{equation*}
        \cdots\arw H^m(E, F^\vee\otimes G(-1, -1))\arw H^m(E, F^\vee\otimes G)\arw \cdots\arw  \Ext_X^m(j_*F, j_*G)\arw \cdots.
    \end{equation*}
\end{lemma}
\begin{proof}
    By adjunction, one has:
    \begin{equation*}
        \Ext_X^m(j_*F, j_*G) \simeq \Ext_E^m(j^*j_*F, G).
    \end{equation*}
    Since $j$ is an embedding of a codimension one variety, it comes with the following distinguished triangle for any $F\in\dbcoh(E)$ (see \cite[Corollary 11.4(ii)]{huybrechts_fm_transform} for details):
    \begin{equation*}
        F\otimes\Oc_E(1, 1)[1]\arw i^*i_* F\arw F.
    \end{equation*}
    It induces a long exact sequence of Ext spaces:
    \begin{equation*}
        \cdots\arw \Ext_E^m(F, G)\arw \cdots\arw \Ext_E^m(i^*i_*F, G)\arw \cdots\arw \Ext_E^{m+1}(F(1, 1), G)\arw\cdots,
    \end{equation*}
    and hence the proof is concluded.
\end{proof}
\begin{lemma}\label{leray_vanishings}
    For $-4\leq a\leq -1$, any sheaf on $E$ of the form $p_\pm^*\Fc\otimes p_\mp^*\Oc(a)$ has no cohomology.
\end{lemma}
\begin{proof}
    By the Leray spectral sequence, we just need to prove that the pushforward of such bundle along $p_{\pm}$ is identically zero. The computation can be handled locally: for every $x\in Y_\pm$ one has:
    \begin{equation*}
        \begin{split}
            \left(p_{\pm*}(p_\pm^*\Fc\otimes p_\mp^*\Oc(a))\right)_x \simeq \left(\Fc\otimes p_{\pm*} p_\mp^*\Oc(a)\right)_x \\
            \simeq H^\bullet(p_\pm^{-1}(x), \Fc_x\otimes p_\mp^*\Oc(a)|_{p_\pm^{-1}(x)} \\
            \simeq \Fc_x\otimes H^\bullet(\PP^4, \Oc(a)) = 0.
        \end{split}
    \end{equation*}
\end{proof}
\begin{lemma}\label{ext_line_bundles_ingredients}
    Consider a pair $(a, b)$ of integers satisfying at least one of these conditions:
    \begin{enumerate}
        \item either $-3\leq a \leq -1$ or $-3\leq b \leq -1$;
        \item $(a, b)$ is either $(-6, 1)$, $(-5, 1)$, $(-4, 1)$, $(-5, 2)$ or $(-4, 2)$.
    \end{enumerate}
    Then one has:
    \begin{equation*}
        \begin{split}
            & H^\bullet(E, j_*\shO(a, b)) =  0, \\
            & H^\bullet(E, j_*\shO(a-1, b-1)) =  0.
        \end{split}
    \end{equation*} 
\end{lemma}
\begin{proof}
    Every bundle as above that satisfies Condition (1) also fulfills the assumptions of Lemma \ref{leray_vanishings}. For the bundles satisfying Condition (2), recall that $\shO(a, b)\simeq \shE_{a\omega_4+b\omega_5}$. Then, we apply the Borel--Weil--Bott theorem to each of them individually.
\end{proof}
\begin{lemma}\label{ext_U_O_ingredients}
    The following vector bundles on $E$, and their twists by $\Oc(-1, -1)$, have no cohomology:
    \begin{enumerate}
        \item $\shU^\vee_+(a, 2)$ for $-3\leq a\leq -2$,
        \item $\shU^\vee_+(a, 1)$ for $-5\leq a\leq -2$,
        \item $\shU^\vee_+(a, 0)$ for $-4\leq a\leq -1$,
        \item $\shU^\vee_+(-1, -1)$ and $\shU^\vee_+(0, -2)$.
    \end{enumerate}
    Moreover, one has:
    \begin{equation*}
        H^\bullet(E, \shU^\vee_+(-1, 1)) = \CC[-1], \hspace{5pt} H^\bullet(E, \shU^\vee_+(-2, 0)) = 0.
    \end{equation*} 
\end{lemma}
\begin{proof}
    By a twist of the dual of the relative Euler sequence \ref{relative_euler}, we can describe the relevant bundles as extensions of, respectively, a twist of $\Uc_4^\vee$ and a line bundle. More precisely, one has:
    \begin{equation*}
        0\arw\Oc(a+1, b-1) \arw \Uc_+^\vee (a, b) \arw \Uc_4^\vee (a, b)\arw 0.
    \end{equation*}
    One can check that the line bundle has no cohomology by Lemma \ref{ext_line_bundles_ingredients}. On the other hand, the contribution to cohomology from the terms of the form $\Uc_4^\vee (ah_++bh_-)$ can be easily computed with the Borel--Weil--Bott theorem, once we recall Equation \ref{weights_U4_shur_powers}.
\end{proof}
\begin{lemma}\label{syms_and_wedges_ingredients}
    The following vector bundles on $E$, and their twists by $\Oc(-1, -1)$ have no cohomology:
    \begin{equation*}
        \begin{split}
            & \Sym^2\Uc_+^\vee(-2, 1), \hspace{5pt} \Sym^2\Uc_+^\vee(-4, 1), \hspace{5pt} \Sym^2\Uc_+^\vee(-5, 1),\\
            & \wedge^2\Uc_+^\vee(-2, 1), \hspace{5pt} \wedge^2\Uc_+^\vee(-4, 1), \hspace{5pt} \wedge^2\Uc_+^\vee(-5, 1).
        \end{split}
    \end{equation*}
\end{lemma}
\begin{proof}
    The short exact sequence \ref{relative_euler} implies the following:
    \begin{equation*}
        0\arw \Oc(2, -2)\arw \Uc_+^\vee(1, -1)\arw\Sym^2\Uc_+^\vee \arw \Sym^2\Uc_4^\vee\arw 0,
    \end{equation*}
    \begin{equation*}
        0\arw\Uc_+^\vee(1, -1)\arw \wedge^2\Uc_+^\vee\arw \wedge^2\Uc_4^\vee\arw 0.
    \end{equation*}
    With appropriate twists of these sequences, we can resolve the six relevant bundles. In the resulting sequences, all terms have already been proven to be acyclic, except for the twists of $\Sym^2\Uc_4^\vee$ and $\wedge^2\Uc_4^\vee$: In light of Equation \ref{weights_U4_shur_powers}, they can be proven to have no cohomology by the Borel--Weil--Bott theorem.
\end{proof}
\begin{lemma}\label{tautological_shour_powers_from_minus_side}
    The following vector bundles on $E$, and their twists by $\Oc(-1, -1)$, have no cohomology:
    \begin{equation*}
        \Uc_-(-2, 2), \Sym^2\Uc_-(-1, 1), \wedge^2\Uc_-(-1, 1).
    \end{equation*}
\end{lemma}
\begin{proof}
    The proof follows exactly as the one of the previous lemma, except for the fact that we use the sequence \ref{relative_euler_minus_side} instead of \ref{relative_euler}.
\end{proof}
\begin{corollary}\label{cor:ext_summing_all_up}
    Consider $F, G \in\dbcoh(E)$. Then, if $F^\vee\otimes G$ satisfies the assumptions of Lemmas \ref{ext_line_bundles_ingredients}, \ref{ext_U_O_ingredients}, \ref{syms_and_wedges_ingredients}, \ref{tautological_shour_powers_from_minus_side}, one has $\Ext_X^\bullet(j_*F, j_*G) = 0$.
\end{corollary}
\begin{proof}
    The proof follows directly by plugging the results of Lemmas \ref{ext_line_bundles_ingredients}, \ref{ext_U_O_ingredients} into the exact sequence of Lemma \ref{we_can_use_koszul}.
\end{proof}
We conclude this section by constructing some exact sequences that will be used in the proof of Theorem \ref{main_theorem_flops}.
First, we recall the \emph{affine tangent bundle} $\wt T $ of $\SS_+$, i.e. the unique extension described as:
\begin{equation}\label{affine_tangent_bundle_spinor_plus}
    0\arw \Oc(-1, 0)\arw \wt T  \arw T(-1, 0)\arw 0.
\end{equation}

\begin{lemma}\label{F_as_an_extension}
    There is a unique extension:
    \begin{equation}\label{sesF}
        0\arw\Uc_+^\vee(-2, 0)\arw\Fc\arw \wt T (0, -1)\arw 0.
    \end{equation}
    Moreover, if we call $\wt T _4$ the affine tangent bundle of $E$, one also has the following short exact sequence:
    \begin{equation*}
        0\arw\Oc(-1, -1)\arw \Fc \arw \wt T _4\arw 0.
    \end{equation*}
\end{lemma}

\begin{proof}
The existence of $\Fc$ as in Equation  \ref{sesF} is a consequence of:
\begin{equation*}
    \RR Hom(\wt T (0, -1), \shU^{\vee}(-2, 0))=\CC[-1].
\end{equation*}
This computation follows from resolving $\wt T ^\vee\otimes\Uc_+^\vee(-2, 1)$ as a tensor product of $\Uc_+^\vee(-2, 1)$ with the dual of the following sequence \cite[Lemma 2.4]{my_paper_with_riccardo}:
\begin{equation}
    0\arw\wt T \arw V_{16}\otimes\Oc\arw \Uc_+(1, 0)\arw 0.
\end{equation}
By using Equation \ref{relative_euler} iteratively, one can eventually resolve each term with homogeneous irreducible vector bundles on $E$. Then, the result follows by the Borel--Weil--Bott theorem.\\
Now, recall the relative tangent bundle sequence of $p_+: OG(4, 10)\longrightarrow \spinor_+$:
\[
\begin{tikzcd}
    0\arrow[r] & T_{rel} \arrow[r] & T_4 \arrow[r] & p_+^*T\arrow[r] &0,
\end{tikzcd}
\]
where one has $T_{rel}\cong \shU_4^{\vee}(-1, 1)$.
Then, there is a commutative diagram:
\[
\begin{tikzcd}
     &0 \arrow[r]\arrow[d] &\shO(-1, -1) \arrow[r]\arrow[d, hook] &\shO(-1, -1) \arrow[r]\arrow[d, hook] &0\\
     0\arrow[r] & \shU_4^{\vee}(0, -2) \arrow[d, equals]\arrow[r, dashed] & \wt T _4  \arrow[d, two heads]\arrow[r, dashed] & \wt T (0, -1)\arrow[d, two heads]\arrow[r] &0\\
    0\arrow[r] & \shU_4^{\vee}(-2, 0) \arrow[r] & T_4(-1, -1)  \arrow[r] & T(-1, -1)\arrow[r] &0.
\end{tikzcd}
\]
All vertical lines and the top and bottom horizontal lines are exact. Hence, the middle line is exact (relative affine tangent sequence). Combining it with the relative Euler sequence yields:
\[
\begin{tikzcd}
    0\arrow[r] & \shO(-1, -1) \arrow[r] \arrow[d, equals]& \shU_+^{\vee}(-2, 0) \arrow[r] \arrow[d, hook]& \shU_4^{\vee}(-2, 0)\arrow[r]\arrow[d, hook] &0\\
     0\ar{r}&\shO(-1, -1) \arrow[r, dashed] \arrow[d]&\shF \arrow[r, dashed]\arrow[d, two heads] &\wt T _4\arrow[r] \arrow[d, two heads]&0\\
& 0\arrow[r] &\wt T (0, -1)\arrow[r, equals] &\wt T (0, -1) \arrow[r] &0.
\end{tikzcd}
\]
Therefore $\shF$ is an extension of $\wt T _4$ by $\shO(-1, -1)$.
\end{proof}
\subsection{The mutations}\label{sec:the_mutations}
Throughout this section, we will use a ``chess-game'' notation to depict full exceptional collections. That is, we will define some symbols to denote specific groups of objects, and the position of such symbols on a grid (a ``chessboard'') will encode the twist by powers of $\Oc(1, 0)$ and $\Oc(0, 1)$. Twists by $\Oc(1, 0)$ will increase towards the right of the grid, while twists by $\Oc(0, 1)$ will increase moving upward. This notation and its name are reminiscent of \cite{thomas_notes_on_hpd}, but the symbols we use are slightly different.\\
\\
We begin by constructing a full exceptional collection for $\SS_+$. Our starting point is an exceptional collection due to Kuznetsov and Polishchuk \cite[Theorem 9.3]{KP_collections_isotropic}, which we slightly mutate to make it more suitable for our setting.
\begin{proposition}\label{mutated_KP_cokllection}
$\dbcoh(\spinor_+)$ admits a full exceptional collection represented in the following chessboard  (mutated KP collection):
\begin{figure}[ht!]
    \centering
    \includegraphics[width=0.5\linewidth]{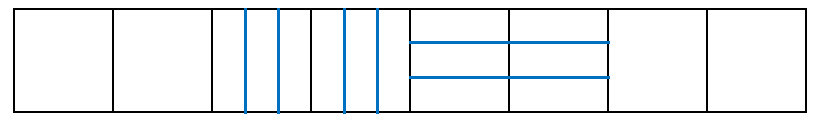}
    \caption{SOD of $\spinor_+$}
    \label{fig:modified_KP}
\end{figure}
\begin{itemize}
\item  The cell \includegraphics[width=0.035\linewidth]{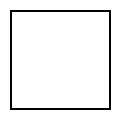} represents $\langle \shO\rangle$.
\item The cell \includegraphics[width=0.035\linewidth]{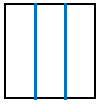}
represents $\langle \wt T ^{\vee}(-1, 0),\shO, \shU_+  ^{\vee}\rangle$.
\item The cell \includegraphics[width=0.035\linewidth]{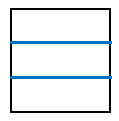}
represents $\langle \shU_+  , \shO, \wt T (1, 0)\rangle$.
\end{itemize}
\end{proposition}
\begin{proof}
    By \cite{my_paper_with_riccardo}, the following KP collection for $\spinor_+$ is full:
    \begin{equation}\label{KP_collection}
        \begin{split}
        \dbcoh(X) = \langle &\Oc, \Oc(1, 0), \Oc(2, 0), \Uc_+^\vee(2, 0), \Sym^2\Uc_+^\vee(2, 0), \Oc(3, 0), \Uc_+^\vee(3, 0), \Sym^2\Uc_+^\vee(3, 0), \\
        & \hspace{40pt}\Oc(4, 0), \Uc_+^\vee(4, 0), \wt T (5, 0), \Oc(5, 0), \Uc_+^\vee(5, 0),  \wt T (6, 0), \Oc(6, 0), \Oc(7, 0) \rangle.
        \end{split}
    \end{equation}
    To obtain the collection of Figure \ref{fig:modified_KP}, we first mutate $\Uc_+ ^\vee(2, 0)$ and $\Uc_+ ^\vee(3, 0)$ one step to the left. In fact, by the tautological exact sequence together with Lemma \ref{ext_U_O_ingredients} and \cite[Lemma A.2]{my_paper_with_riccardo} one easily sees that $\LL_{\Oc}\Uc_+^\vee \simeq\Uc_+$. Then, we need to mutate the bundles $\Sym^2\Uc_+ ^\vee(2, 0)$ and $\Sym^2\Uc_+ ^\vee(3, 0)$ two steps to the left. One has the following long exact sequence \cite[Equation 2.12]{my_paper_with_riccardo}:
    \begin{equation}\label{eq:partial_big_sequence}
        0\arw\wt T ^\vee(-1, 0)\arw \VV^{D_5}_{\omega_1}\otimes\Uc_+\arw \VV^{D_5}_{2\omega_1}\otimes\Oc\arw \Sym^2\Uc_+^\vee\arw 0.
    \end{equation}
    Hence, by mutating $\Sym^2\Uc_+ ^\vee(2, 0)$ one step to the left, we obtain the kernel of the last map of the sequence \ref{eq:partial_big_sequence} twisted by $\Oc(2, 0)$, and mutating it again one step to the left we find the desired object. All the relevant Exts have been addressed in \cite[Proof of Proposition 3.8]{my_paper_with_riccardo} and local references.
\end{proof}
\subsubsection{Proof of Theorem \ref{main_theorem_flops}}
By Orlov's blow-up formula, we have the two following SODs for $\dbcoh(\wt X)$:
\begin{equation}\label{sod_from_orlov}
    \begin{split}
        \dbcoh(\blX)=\langle  \pi_{\pm}^*D(X_{\pm}), &j_*(p_{\pm}^*D(\spinor_{\pm}), j_*(p_{\pm}^*D(\spinor_{\pm})(h_{\mp}), \\
        &j_*(p_{\pm}^*D(\spinor_{\pm})(2h_{\mp}), j_*(p_{\pm}^*D(\spinor_{\pm})(3h_{\mp})\rangle.
    \end{split}
\end{equation}
Using the collection of Proposition \ref{mutated_KP_cokllection}, the SOD of $^{\perp}D(X_+)$ can be represented on the chessboard of Figure \ref{fig:modified_KP}.
Hereafter, we describe the main steps of the proof by illustrating a sequence of mutations, which identifies the semiorthogonal complements of $\dbcoh(X_+)$ and $\Phi\dbcoh(X_-)$ as admissible subcategories of $\dbcoh(\wt X)$, where $\Phi$ is a suitable equivalence of categories determined by the mutations. While the proof is explained for the case of a flop of type $D_5$ ``over a point'', i.e., with exceptional divisor isomorphic to $OG(4, V_{10})$, the same pattern of mutations will be used for the proof of the relative case of Theorem \ref{main_theorem_flops} and Theorem \ref{main_theorem_pairs} over any base.

The starting point is the decomposition depicted in Figure \ref{fig:starting_point}.
\begin{figure}[ht!]
    \centering
    \includegraphics[width=0.5\linewidth]{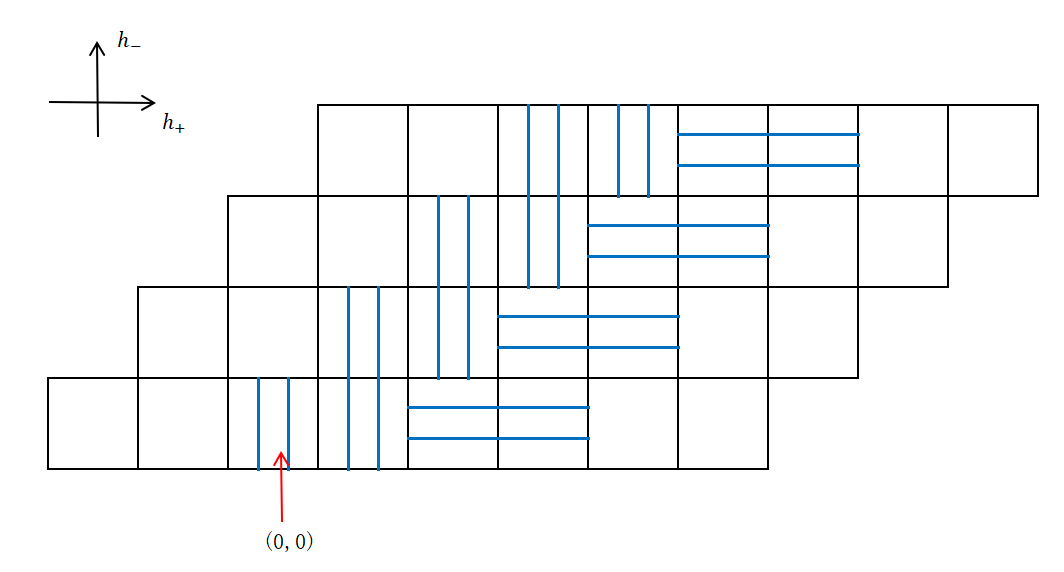}
    \caption{SOD of $^{\perp}\dbcoh(X_+)$}
    \label{fig:starting_point}
\end{figure}

Note that all exceptional objects generating $^{\perp}D(X_+)$ are twists of vector bundles pushed forward to the exceptional divisor. We will therefore omit the functors $p^*$, $q^*$, and $j_*$ for the sake of a clearer exposition.
\begin{description}
\item[Step 1] The first step is to apply the Serre functor to the boxes \includegraphics[width=0.035\linewidth]{O.png} in positions $(7,3)$ and $(8,3)$: They are twisted by the canonical bundle and sent to the beginning of the collection, as white boxes at positions $(3, -1)$ and $(4, -1)$. Then, we mutate $D(X_+)$ to the beginning of the SOD, where it is replaced by
$$\Phi_1 \dbcoh(X_+): = \LL_{\langle  \Oc(3, -1),  \Oc(4, -1)\rangle} \dbcoh(X_+).$$
Also, note that by Lemma \ref{ext_line_bundles_ingredients} and Corollary \ref{cor:ext_summing_all_up} the boxes \includegraphics[width=0.035\linewidth]{O.png} at $(-2,0)$ and $(-1,0)$ are orthogonal to the ones at positions $(3, -1)$ and $(4, -1)$: hence the former can be right-mutated to the far left, and then to the end of the collection by applying the inverse Serre functor, after mutating $\Phi_1\dbcoh(X_+)$ two steps to the right. These moves are described by the red arrows in Figure \ref{m1}, where the chessboard describes the right semiorthogonal complement of $\Phi_2\dbcoh(X)$, with
\begin{equation*}
    \Phi_2 = \RR_{\langle \Oc(-2, 0), \Oc(-1, 0)\rangle} \circ \Phi_1.
\end{equation*}
\begin{figure}[ht!]
    \centering
    \includegraphics[width=0.5\linewidth]{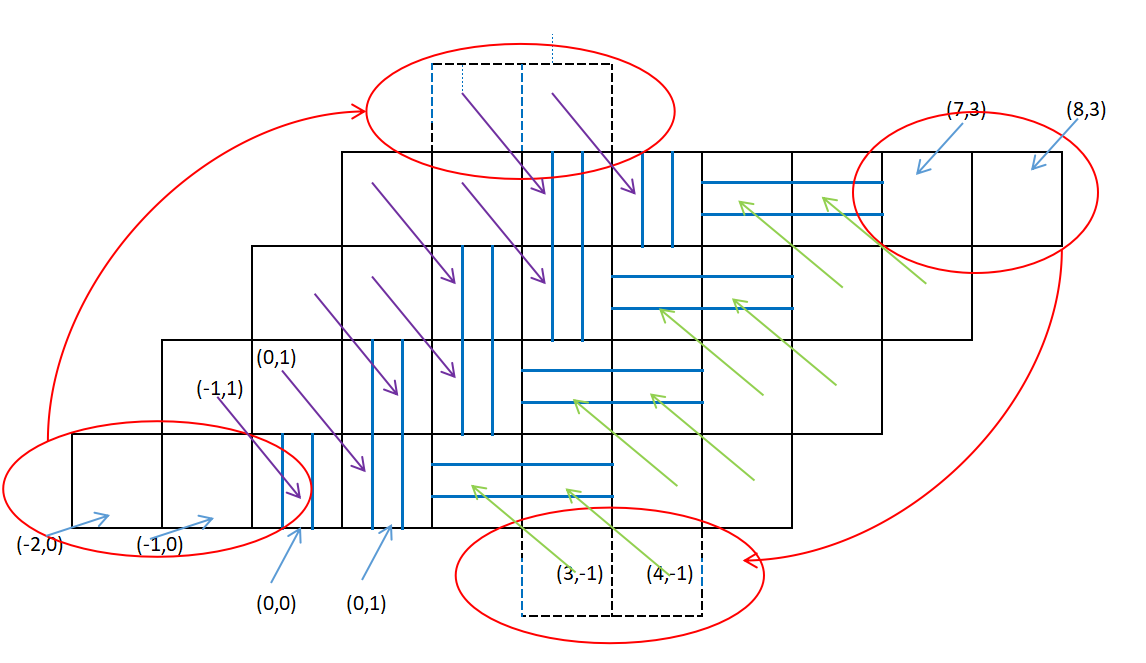}
    \caption{Step 1 and Step 2}
    \label{m1}
\end{figure}
\item[Step 2] Insert blank cells, along the green arrows, into the blocks of shape \includegraphics[width=0.035\linewidth]{dualblock.png}, as depicted in Figure \ref{m1}. For instance, insert \includegraphics[width=0.035\linewidth]{O.png} at $(4,-1)$ and $(3,-1)$ into the cells at $(3,0)$ and $(2,0)$ respectively, which requires to prove that the blocks $\langle  \Oc(2, -1)\rangle$ and $\langle  \wt T ^{\vee}(-3, 0),$ $ \shO(-2, 0),  \shU_+^{\vee}(-2, 0),   \wt T ^{\vee}(-2, 0),  \shO(-1, 0),  \shU_+^{\vee}(-1, 0),  \shU_+,$\\ $ \shO,  \wt T (1, 0)\rangle$ are orthogonal. Again, we use Corollary \ref{cor:ext_summing_all_up} to show that $\langle  \Oc(2, -1)\rangle$ is orthogonal to the twists of $ \Oc$ and $ \Uc_+^\vee$. We also observe that $ \Uc_+\in\langle  \Oc,  \Uc_+^\vee\rangle$ and therefore $\langle  \Oc(2, -1)\rangle$  is orthogonal to $ \Uc_+$ because it is orthogonal to both $ \Oc$ and $ \Uc_+^\vee$. To prove that $ \Oc(2, -1)$ is orthogonal to the twists of $\wt T $ and $\wt T ^\vee$, we use the appropriate twists and duals of the sequence \ref{affine_tangent_bundle_spinor_plus} and Corollary \ref{cor:ext_summing_all_up}.\\
\\
Similarly, insert the cells at $(-1, 1)$ and $(0, 1)$ into the cells at (0,0) and (0,1) respectively, as depicted by the purple arrows in Figure \ref{m1}. The computation of the relevant vanishings is identical to the previous one; hence, it will be omitted.
\item[Step 3] Left-mutate the cells at $(5, 3)$ and $(6, 3)$ to the far left (red arrow in Figure \ref{m2}), and then left-mutate $\Phi_2\dbcoh(X_+)$ through them. We will denote the result by $\Phi_3\dbcoh(X_+)$, where:
\begin{equation*}
    \Phi_3 = \LL_{\langle  \Oc(2, -2),  \shU_+(1, -1)  ,  \shO(1, -1),  \wt T (2, -1),  \Oc(2, -1),  \shU_+(1, 0)  ,  \shO(1, 0),  \wt T (2, 0)\rangle}\circ\Phi_2.
\end{equation*}
Now, let us move the cells according to the orange arrows in Figure \ref{m2}:
\begin{figure}[ht!]
    \centering
    \includegraphics[width=0.35\linewidth]{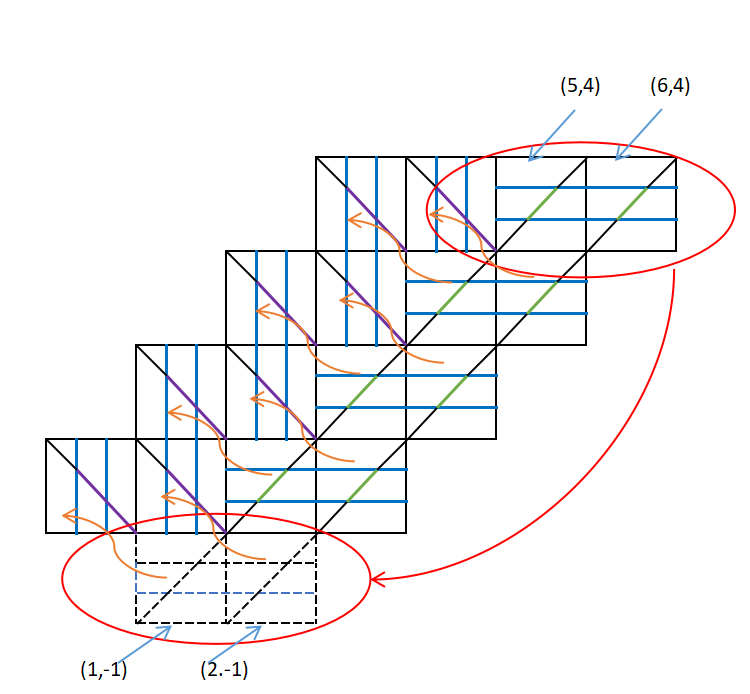}
    \caption{Step 3}
    \label{m2}
\end{figure}
This again requires computing some vanishings. More precisely, we need to show the following orthogonality condition up to global twists:
\begin{equation*}
    \begin{split}
        \langle  \Oc(1, -1),  \shU_+  ,  \shO,  \wt T (1, 0)\rangle \perp
        \langle  &  \wt T ^{\vee}(-3, 1),   \shO(-2, 1),\\
        &  \shU_+^{\vee}(-2, 1),   \Oc(-3, 2)\rangle.    
    \end{split}
\end{equation*}
The proof is essentially as the one of the previous vanishing, except for the term $\wt T ^\vee\otimes\wt T ^\vee$, which can be resolved by taking the tensor product of the sequence \ref{affine_tangent_bundle_spinor_plus} with $\wt T ^\vee$, and then resolving the first and second term with \ref{affine_tangent_bundle_spinor_plus}. The resulting objects can be handled with Corollary \ref{cor:ext_summing_all_up}.\\
\\
We end up with a collection given by twists of the new block $\Cc := $\includegraphics[width=0.035\linewidth]{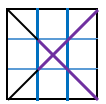}, where
\begin{equation*}
    \begin{split}
        \shC= \langle &  \shO(2, -2),  \shU_+(1, -1),  \shO(1, -1), \\
        &  \wt T (2, -1),  \wt T ^{\vee}(-1, 0),  \shO,  \shU_+^{\vee},  \shO(-1, 1)\rangle.
    \end{split}
\end{equation*}
\item[Step 4] After an application of the Serre functor (red arrow), the collection is represented by the chessboard of Figure \ref{m3}:
\begin{figure}[ht!]
    \centering
    \includegraphics[width=0.35\linewidth]{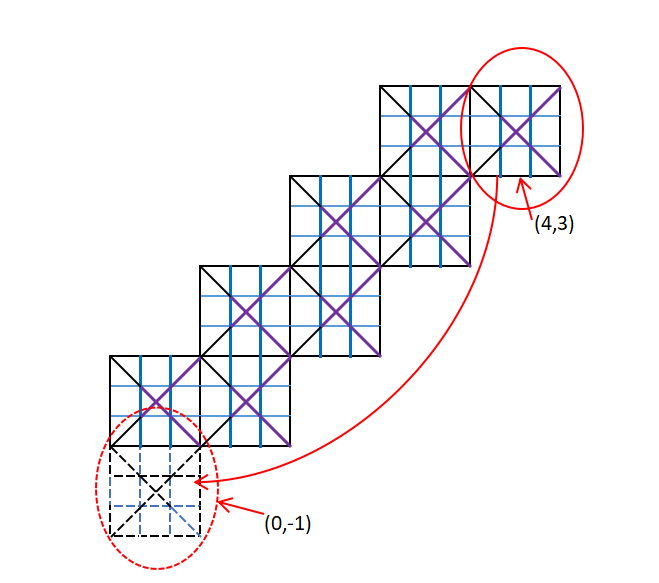}
    \caption{Step 4}
    \label{m3}
\end{figure}
\item[Step 4.1] The categories $\langle\wt T (2, -1)\rangle$ and $\langle \wt T ^{\vee}(-1, 0), \shO\rangle$ are orthogonal to each other. In fact, in light of Lemma \ref{we_can_use_koszul}, we just need to prove that the bundles $\wt T ^\vee\otimes\wt T ^\vee(-3, 1)$, $\wt T ^\vee(-2, 1)$ on $E$ and their twists by $\Oc(-1, -1)$ have no cohomology: by means of (appropriate duals and twists of) the sequences \ref{affine_tangent_bundle_spinor_plus} and \ref{relative_euler}, we can easily reduce the computation to the cohomology of twists of symmetric and wedge powers of $\Uc_4^\vee$, which have already been proven to have no cohomology (see Corollary \ref{cor:ext_summing_all_up}). Similarly, we can exchange $\wt T ^{\vee}(-1,0)$ and $\shO(1, -1)$ because the relevant Ext spaces are exactly the same which appear in the previous orthogonality check. The resulting expression for $\Cc$ is:
\begin{equation*}
    \begin{split}
        \shC=\langle &  \shO(2, -2),  \shU_+(1, -1),  \wt T ^{\vee}(-1, 0),  \shO(1, -1), \\
        &  \shO,  \wt T (2, -1),  \shU_+^{\vee},  \shO(-1, 1)\rangle.
    \end{split}
\end{equation*}
\item[Step 4.2]
Next, observe that by Corollary \ref{cor:ext_summing_all_up} and the usual approach we can right-mutate $\wt T (2, -1)$ through $\shU_+^{\vee}$, and left-mutate $\wt T ^{\vee}(-1, 0)$ through $\shU_+(1, -1)$. Then, in light of Lemma \ref{F_as_an_extension}, the resulting expression for $\Cc$ is:
\begin{equation*}
    \begin{split}
        \shC=\langle & \shO(2, -2),\shF^{\vee}(-1, -1),\shU_+(1, -1),\\
        & \shO(1, -1), \shO, \shU_+^{\vee}, \shF(2, 0),\shO(-1, 1)\rangle,
    \end{split}
\end{equation*}
where the mutation yielding $\Fc^\vee$ follows from the dual of the sequence \ref{sesF}. 
In the following steps, we will perform additional mutations inside $\Cc$. We will use the sequences \ref{relative_euler}, \ref{sesF} and their duals and twists to reduce the relevant computations to the cohomology of irreducible homogeneous vector bundles. Since these operations are identical to those we illustrated in the previous step, with no additional technique introduced, we will omit the explicit computations.

\item[Step 4.3] By Lemma \ref{ext_line_bundles_ingredients}, we see that $\shO$ and $\shO(1, -1)$ are orthogonal: let us exchange them, and then right-mutate $\shU_+(1, -1)$ through $\shO$ and left-mutate $\shU_+^{\vee}$ through $\shO(1, -1)$. The last two mutations follow from the fact that:
\begin{equation}
    \Ext^\bullet_{\wt X}(\Uc_+(1, -1), \Oc) = \Ext^\bullet_{\wt X}(\Oc(-1, 1), \Uc_+^\vee) = \CC[0],
\end{equation}
together with the sequence \ref{relative_euler} and its dual. We get:
$$\shC=\langle \shO(2, -2),\shF^{\vee}(-1, -1), \shO,\shU_4(1, -1),\shU_4^{\vee},\shO(1, -1), \shF(2, 0),\shO(-1, 1)\rangle.$$
\item[Step 4.4] By the dual of \ref{sesF} and Lemmas \ref{ext_line_bundles_ingredients}, \ref{ext_U_O_ingredients}, \ref{cor:ext_summing_all_up} we see that we can exchange $\shO(2, -2)$ and $\langle \shF^{\vee}(-1, -1),\shO\rangle$. Then, observe that:
\begin{equation*}
    \Ext_{\wt X}^\bullet(\Uc_4(1, -1), \Oc(2, -2)) = \CC[-1].
\end{equation*}
This, together with the sequence \ref{relative_euler_minus_side}, allows to left-mutate $\shU_4(1, -1)$ through $\shO(2, -2)$ obtaining $\Uc_-(1, -1)$. After a similar operation to the objects $\shU_4^{\vee}$, $\shF(2, 0)$, $\shO(-1, 1)$ we find:
$$\shC=\langle \shF^{\vee}(-1, -1), \shO,\Uc_- (1, -1),\shO(2, -2),\shO(-1, 1),\Uc_4^{\vee},\shO(1, -1), \shF(2, 0)\rangle.$$
\item[Step 4.5] Exchange $\shO(2, -2)$ and $\langle \shO(-h_+ h_-),\Uc_-^\vee , \shO(1, -1)\rangle$, and then exchange $\shO(-1, 1)$ and $\langle \shO, \Uc_- (1, -1)\rangle$:
$$\shC=\langle \shF^{\vee}(-1, -1),\shO(-1, 1), \shO,\Uc_- (1, -1),\Uc_-^\vee ,\shO(1, -1),\shO(2, -2),\shF(2, 0)\rangle.$$
% \item[Step 4.6] Exchange $\Uc_- (1, -1)$ and $\Uc_-^\vee $, and then left-mutate $\Uc_-^\vee $ through $\langle\shO(-1, 1),\shO\rangle$ and similarly for $\Uc_- (1, -1)$:
% $$\shC=\langle \shF^{\vee}(-1, -1),\shQ^{\vee},\shO(-1, 1), \shO,\shO(1, -1),\shO(2, -2),\shQ^{\vee}(1, -1),\shF(2, 0)\rangle$$
% \item[Step 4.7] Exchange $\shO$ and $\shO(1, -1)$:
% $$\shC=\langle \shF^{\vee}(-1, -1),\shQ^{\vee},\shO(-1, 1), \shO(1, -1),\shO,\shO(2, -2),\shQ(1, -1),\shF(2, 0)\rangle$$
% Write $\shA=\langle \shF^{\vee}(-1, -1),\shQ^{\vee},\shO(-1, 1), \shO(1, -1)\rangle$, and then 
% $$\shC=\langle \shA, \shA^{\vee}(1, -1)\rangle$$

% \item[Step 4.3] Exchange $\shO(2, -2)$ and $\shF^{\vee}(-1, -1)$, and then left-mutate $\shU(1, -1)$ through $\shO(2, -2)$. Similarly for $\shU^{\vee}, \shF(2, 0),\shO(-1, 1)$:
% $$\shC=\langle \shF^{\vee}(-1, -1),\shQ^{\vee}(1, -1), \shO(2, -2),\shO(1, -1), \shO, \shO(-1, 1),\shQ, \shF(2, 0),\rangle$$

% \item[Step 4.4] Since all $\shO(2, -2),\shO(1, -1), \shO, \shO(-1, 1)$ are orthogonal to each other, we change them into the following:
% $$\shC=\langle \shF^{\vee}(-1, -1),\shQ^{\vee}(1, -1), \shO(-1, 1),\shO(1, -1),\shO,\shO(2, -2),\shQ, \shF(2, 0)\rangle$$

% $\langle \shO(2, -2),\shF^{\vee}(-1, -1),\shU(1, -1)\rangle=\langle \shF^{\vee}(-1, -1),\shQ^{\vee}(1, -1),\shO(2, -2)$

% Exchange $\shO(1, -1)$ and $\shO$, and then mutate $\shU(1, -1)$ through $\shO$ and mutate $\shU^{\vee}$ through $\shO(1, -1)$:
% $$\shC=\langle \shO(2, -2),\shF^{\vee}(-1, -1),\shO,\shU_4(1, -1), \shU_4^{\vee}, \shO(1, -1), \shF(2, 0),\shO(-1, 1)\rangle$$
\item[Step 4.6] Exchange $\Uc_- (1, -1)$ and $\Uc_-^\vee $. 
$$\shC=\langle \shF^{\vee}(-1, -1),\shO(-1, 1),\shO,\Uc_-^\vee ,\Uc_- (1, -1),\shO(1, -1),\shO(2, -2),\shF(2, 0)\rangle.$$
Write $\shA=\langle \shF^{\vee}(-1, -1),\shO(-1, 1),\shO,\shU^{\vee}\rangle$, and then
% \begin{equation*}
%         \shC=\langle \shA, \shA^{\vee}(1, -1)\rangle:=
% \begin{minipage}{0.1\textwidth}
%     \centering
%     \includegraphics{block1.png}
% \end{minipage}
% \end{equation*}
% and then left(resp. right) mutate $\Uc_-^\vee $ (resp. $\shU_-(1, -1)$) through $\shO$(resp. $\shO(1, -1)$):
% $$\shC=\langle \shF^{\vee}(-1, -1),\shO(-1, 1),\Uc_- ,\shO,\shO(1, -1),\Uc_-^\vee (1, -1),\shO(2, -2),\shF(2, 0)\rangle$$
% Write $\shA=\langle \shF^{\vee}(-1, -1),\shO(-1, 1),\Uc_- ,\shO\rangle$. 
$$\shC=\langle \shA, \shA^{\vee}(1, -1)\rangle:=\includegraphics[width=0.2\linewidth]{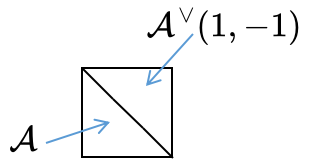}.$$
The right orthogonal of $\Phi_3\dbcoh(X_+)$ can be represented by the following:
\begin{figure}[ht!]
    \centering
    \includegraphics[width=0.35\linewidth]{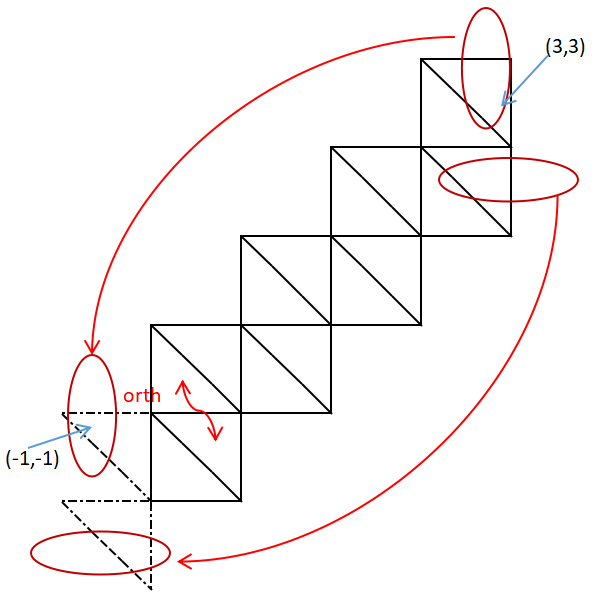}
    \caption{Step $5.1$}
    \label{m4}
\end{figure}
\end{description}
Observe that, by a simple but tedious computation, the blocks $\Ac$ and $\Ac(1, -2)$ are orthogonal.
\begin{description}
\item[Step 5.1] Left-mutate the upper half cells at $(3,2)$ and $(3,3)$ to the far left, as depicted by the long red arrows of Figure \ref{m4}. 
We introduce the following new block:
$$\shB=\langle \shA^{\vee}(0, -3), \shA^{\vee}(0, -2),\shA(0, -1),\shA\rangle=\includegraphics[width=0.1\linewidth]{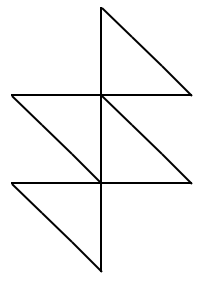}.$$
With this notation, one has:
\begin{equation*}
    ^{\perp}\Phi_4\dbcoh(X_+)\cong \langle \shB, \shB(1, 1), \shB(2, 2),\shB(3, 3)\rangle,
\end{equation*}
where:
\begin{equation*}
    \Phi_4 = \LL_{\langle\Ac(-1, -2), \Ac(-1, -1)\rangle}\circ\Phi_3.
\end{equation*}
\item[Step 5.2] The next step consists of performing some additional mutations inside the block $\shB$ in order to obtain pullbacks from $X_-$. Let us write $\Bc$ explicitly:
\begin{align*}
 \shB=\langle &\boxed{\shU_-(0, -3),\shO(0, -3),\shO(1, -4), \shF(1, -2),\shU_-(0, -2), \shO(0, -2)},\shO(1, -3),\\ & \shF(1, -1), \shF^{\vee}(-1, -2),\shO(-1, 0),\shO(0, -1),\Uc_-^\vee (0, -1),\shF^{\vee}(-1, -1),\shO(-1, 1),\shO,\Uc_-^\vee \rangle,
\end{align*}
and let us focus on the rectangular box.
Recall that $\shF(1, -2)$ fits into the short exact sequence (\ref{sesF}) (using the projection onto $\spinor_-$):
\[
\begin{tikzcd}
    0\arrow[r] & \Uc_-^\vee (1, -4) \arrow[r] &\shF(1, -2) \arrow[r] &\wt T ^-(0, -2) \arrow[r] &0.
\end{tikzcd}
\]
Also recall that $\wt T ^-(0, -2)$ fits into:
\[
\begin{tikzcd}
    0\arrow[r] & \wt T ^-(0, -2) \arrow[r] &\shO(0, -2)\otimes V_{16} \arrow[r] &\shU_-(0, -3) \arrow[r] &0,
\end{tikzcd}
\]
and the Euler sequence:
\[
\begin{tikzcd}
    0\arrow[r] & \shU_-(1, -4) \arrow[r] &\shO(1, -4)\otimes V_{10} \arrow[r] &\Uc_-^\vee (1, -4) \arrow[r] &0.
\end{tikzcd}
\]
Then the rectangular box is equivalent to
$$\langle\shU_-(0, -3),\shO(0, -3),\shU_-(0, -2), \shO(0, -2),\Uc_- (1, -4),\shO(1, -4)\rangle.$$
Similar mutations can be applied to $\shF(1, -1), \shF^{\vee}(-1, -2)$ and $\shF^{\vee}(-1, -1)$ in $\shB$. If we adopt the notation $\includegraphics[width=0.035\linewidth]{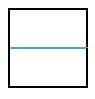}=\langle \shU_-, \shO\rangle$, we get:
\begin{figure}[ht!]
    \centering
    \includegraphics[width=0.3\linewidth]{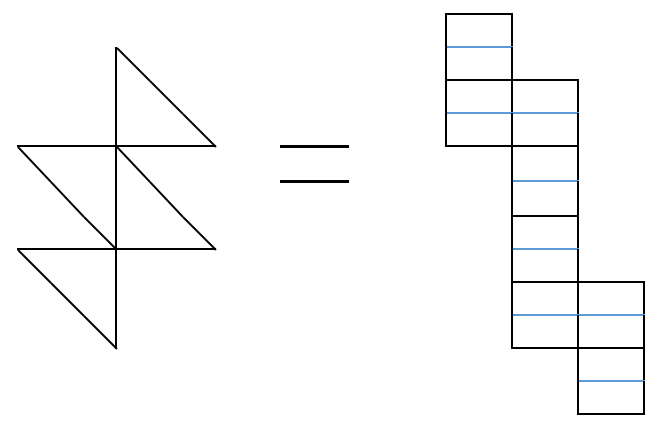}
    \caption{Mutation inside $\shB$ ($\includegraphics[width=0.035\linewidth]{b4.png}=\langle \shU_-, \shO\rangle$)}
    \label{fig:enter-label}
\end{figure}
% Here, $\includegraphics[width=0.035\linewidth]{b4.png}=\langle \shU_-, \shO\rangle$
\item[Step 5.3] By plugging Kuznetsov's collection into the SOD \ref{sod_from_orlov} (minus side), we find the chessboard for $^\perp\dbcoh(X_-)$ in Figure \ref{fig:last_serre_functor} (omitting the dashed boxes and arrows). 
\begin{figure}[h!]
    \centering
    \includegraphics[width=0.3\linewidth]{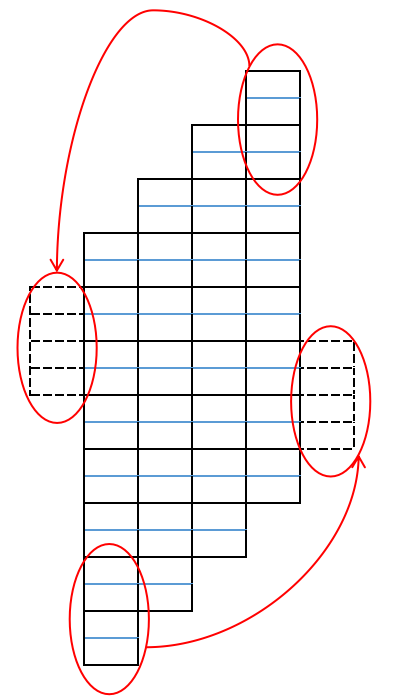}
    \caption{Final mutation}
    \label{fig:last_serre_functor}
\end{figure}
If we now apply the Serre functor as depicted by the red arrows in Figure \ref{fig:last_serre_functor}, we obtain the same chessboard that describes $^\perp\Phi_4\dbcoh(X_+)$, and it concludes the proof of Theorem \ref{main_theorem_flops}.
\end{description}

\subsection{Flops of type \texorpdfstring{$D_5$}{} over a base}\label{sec:relative_flops}
In this section, we show that the proof of Theorem \ref{main_theorem_flops} holds for a broader class of flops. 
Let $G :=\operatorname{Spin}(V_{10})$ be the simply connected algebraic group of type $D_5$, and $P_{4,5}$ be the parabolic subgroup of $G$ associated to the fourth and fifth simple roots so that $OG(4, V_{10}) = G/P_{4,5}$ holds. Given a principal $G$-bundle ($G$-torsor) $\Uc_4$ on $B$, one has the associated locally trivial fibre bundle $\Pc\times^G G/P_{4,5}$ over $B$, where the notation $\times^G$ denotes the quotient of the product by the equivalence relation $(g.t, v)\simeq (t, g.v)$ for all $g\in G$ and $(t, v)\in \Uc_4\times G/P_{4,5}$. Recall that the spinor varieties of $D_5$ are also $G$-homogeneous: $\spinor_+=G/P_4$, $\spinor_-=G/P_5$, and there are associated locally trivial fibre bundles $\Pc\times^G G/P_4$ and $\Pc\times^G G/P_5$ over $B$. Similarly to (\ref{diagram_D5_roof}), there is a roof structure on $\Pc\times^G G/P_{4,5}$ :
\begin{equation*}
\begin{tikzcd}[row sep = huge]
                        & \Pc\times^G G/P_{4,5}  \arrow[dl, swap, "p_+"] \arrow[dr,"p_-"] \\
                         \Pc\times^G G/P_4 &  &\Pc\times^G G/P_5.                 
\end{tikzcd}
\end{equation*}
Then there is a simple flop $\shX_+\dashrightarrow\shX_-$ 
with the exceptional divisor isomorphic to $\Pc\times^G G/P_{4,5}$. This flop is resolved by blowups of $\shX_\pm$ in smooth centers of codimension five, isomorphic to $\Pc\times^G \spinor_{\pm}$.
We can adapt Diagram \ref{diagram_D5_flop_over_a_point} to the present setting:
\begin{center}
\begin{tikzcd}[row sep = huge]\label{diagram_D5_flop_over_a_base}
&&\Pc\times^G G/P_{4,5} \arrow[ddll,swap,"p_+"] \arrow[d, hook,"j"]\arrow[ddrr,"p_-"] \\
&&  \arrow[dl,swap,"\pi_+"]  \widetilde{\shX} \arrow[dr,"\pi_-"]\\
  \Pc\times^G \spinor_{+} \arrow[r,hook,"i_+"]& \shX_+ \arrow[rr,dashed, "\mu"]    & & \shX_- & \arrow[l, hook',swap, "i_-"]\Pc\times^G \spinor_{-}. 
\end{tikzcd}
\end{center}
% Namely, consider a vector bundle $\Ec\arw B$ of rank $10$ on a smooth projective variety $B$, and the simple flop $X_+\dashrightarrow X_-$ with exceptional divisor isomorphic to the orthogonal Grassmann bundle $\Oc\Gc(4, \Ec)$. This flop is resolved by blowups of $X_\pm$ in smooth centers of codimension five, isomorphic to the orthogonal Grassmann bundles $\Oc\Gc(5,\Ec)^\pm$.
% We can adapt Diagram \ref{diagram_D5_flop_over_a_point} to the present setting:
% \begin{center}
% \begin{tikzcd}[row sep = huge]\label{diagram_D5_flop_over_a_base}
% &&E=\Oc\Gc(4,\Ec) \arrow[ddll,swap,"p_+"] \arrow[d, hook,"j"]\arrow[ddrr,"p_-"] \\
% &&  \arrow[dl,swap,"\pi_+"]  \widetilde{X} \arrow[dr,"\pi_-"]\\
%   \Oc\Gc_+(5,\Ec) \arrow[r,hook,"i_+"]& X_+ \arrow[rr,dashed, "\mu"]    & & X_- & \arrow[l, hook',swap, "i_-"]\Oc\Gc_-(5,10) 
% \end{tikzcd}
% \end{center}
We can construct an exact functor $\mathfrak F$ sending homogeneous vector bundles over $OG(4, V_{10})$ to vector bundles over $\Pc\times^G G/P_{4,5}$. This equivalence factors through the category of $P_{4,5}$-modules (where the functor $F$ maps homogeneous vector bundles to the associated modules):
        \begin{equation}\label{eq_relativizationfunctor}
            \begin{tikzcd}[column sep = huge, row sep = large]
                \operatorname{Vect}^{P_{4,5}}(G/P_{4,5}) \ar{rr}{\mathfrak F:=\Pc\times^G G\times^{P_{4,5}} (-) \circ F}\ar[swap]{dr}{F} & & \operatorname{Vect}(\Pc\times^G G/P_{4,5})\\
                & P_{4,5}-\operatorname{Mod}\ar[swap]{ur}{\Pc\times^G G\times^{P_{1,2}} (-)} &.
            \end{tikzcd}
        \end{equation}
For the details of this construction, we refer to \cite{mypaper_relative_grassmann_flips} and the references therein.
Recall that, by applying \cite[Theorem 3.1]{samokhin}, given a flag bundle $\pi:\Pc\times^G G/P\arw B$ with $B$ smooth projective, and a full exceptional collection $\dbcoh(G/P) = \langle V_1,\dots, V_n\rangle$, one has a semiorthogonal decomposition:
\begin{equation*}
    \dbcoh(\Pc\times^G G/P) = \langle \mathfrak F(V_1)\otimes\pi^*\dbcoh(B), \dots, \mathfrak F(V_n)\otimes\pi^*\dbcoh(B) \rangle.
\end{equation*}
The following result will be presented in a slightly more general setting than what we need for the purpose of this work. We will focus on left mutations, since the proof for right mutations is identical.
\begin{lemma}\label{lem:mutations}
    Let $G/P$ be a rational homogeneous variety, and $X = \Pc\times^G G/P$ a flag bundle over a smooth projective variety $B$. Consider a triple $V_1, V_2, V_3$ of homogeneous, exceptional vector bundles (though not necessarily homogeneous irreducible) on $G/P$, such that they satisfy the following relation in $\dbcoh(G/P)$:
    \begin{equation*}
        \LL_{V_1} V_2 \simeq V_3.
    \end{equation*}
    Then, one also has:
    \begin{equation*}
        \LL_{\mathfrak F(V_i)} \mathfrak F(V_j) \simeq \mathfrak F(V_k).
    \end{equation*}
\end{lemma}
\begin{proof}
    Consider the functor:
    \begin{equation*}
        \begin{tikzcd}[row sep = tiny, column sep = large, /tikz/column 1/.append style={anchor=base east} ,/tikz/column 2/.append style={anchor=base west}]
            f_{V_i}: \dbcoh(B) \ar[hook]{r} & \dbcoh(X)\\
            E \ar[maps to]{r} & \pi^* E \otimes \mathfrak F(V_i).
        \end{tikzcd}
    \end{equation*}
    Then, the left mutation of $V_j$ through $V_i$ is defined by the following distinguished triangle:
    \begin{equation}\label{mutation_triangle}
        f_{V_i}f_{V_i}^! \mathfrak F(V_j) \arw \mathfrak F(V_j) \arw \LL_{\mathfrak F(V_i)} \mathfrak F(V_j) \arw f_{V_i}f_{V_i}^! \mathfrak F(V_j)[1]
    \end{equation}
    where $f_{V_i}^!$ is the right adjoint functor of $f_{V_i}$. Explicitly, one has:
    \begin{equation*}
        \begin{tikzcd}[row sep = tiny, column sep = large, /tikz/column 1/.append style={anchor=base east} ,/tikz/column 2/.append style={anchor=base west}]
            f_{V_i}^!: \dbcoh(X) \ar{r} & \dbcoh(B)\\
            W \ar[maps to]{r} & \pi_* R\Hc om_X(\mathfrak F(V_i), W).
        \end{tikzcd}
    \end{equation*}
    Evaluating the latter functor on $\mathfrak F(V_j)$ yields:
    \begin{equation*}
        \begin{split}
            \pi_* R\Hc om_X(\mathfrak F(V_i), \mathfrak F(V_j)) & \simeq H^\bullet(G/P, R\Hc om_{G/P}(V_i, \mathfrak F^! \mathfrak F(V_j))) \\
            & \simeq H^\bullet(G/P, R\Hc om_{G/P}(V_i, V_j)) \\
            & \simeq R Hom_{G/P}(V_i, V_j). \\
        \end{split}
    \end{equation*}
     Plugging this result into the triangle \ref{mutation_triangle} gives:
     \begin{equation}\label{mutation_triangle_II}
        R Hom_{G/P}(V_i, V_j) \otimes \mathfrak F(V_j) \arw \mathfrak F(V_j) \arw \LL_{\mathfrak F(V_i)} \mathfrak F(V_j) \arw R Hom_{G/P}(V_i, V_j) \otimes \mathfrak F(V_j)[1].
    \end{equation}
    On the other hand, consider the triangle defining $\LL_{V_i} V_j$ in $\dbcoh(G/P)$:
    \begin{equation}\label{mutation_triangle_III}
        R Hom_{G/P}(V_i, V_j)\otimes V_j \arw V_j \arw \LL_{V_i} V_j \arw R Hom_{G/P}(V_i, V_j)\otimes V_)[1].
    \end{equation}
    The proof is completed once we note that the triangle \ref{mutation_triangle_II} is the result of applying the (triangulated) functor $\mathfrak F$ to the triangle \ref{mutation_triangle_III}.
\end{proof}
\subsubsection{Proof of Theorem \ref{main_theorem_flops} -- relative case}\label{sec:proof_theorem_flops_relative_case}
The sequence of mutations described in Section \ref{sec:the_mutations} can be applied verbatim to the relative case. In fact, any of such mutations is of the form $\LL_{\iota_*V_i} \iota_*V_j \simeq \iota_*V_k$ or $\RR_{\iota_*V_i} \iota_*V_j \simeq \iota_*V_k$. By applying the functor $\iota_*$ to the triangle \ref{mutation_triangle}, together with Lemma \ref{we_can_use_koszul} and Lemma \ref{lem:mutations}, we see that for all the mutations we consider, the isomorphism $\LL_{\iota_*\mathfrak F(V_i)} \iota_*\mathfrak F(V_j) \simeq \iota_*\LL_{\mathfrak F(V_i)}\mathfrak F(V_j)$ holds, and the right-hand side, again by Lemma \ref{lem:mutations}, is isomorphic to $\mathfrak F(\LL_{V_i} V_j)$. We can therefore proceed with the sequence of mutations described in Section \ref{sec:the_mutations}.

\section{Calabi--Yau fivefolds of type \texorpdfstring{$D_5$}{}}\label{sec:calabi_yau_5folds}
Consider a general section $s\in H^0(OG(4, V_{10}), \Oc(1, 1))$. We can associate to $s$ a pair $(Y_+, Y_-)$ of varieties defined by $Y_\pm:= Z(p_{\pm*} s)\subset \spinor_\pm$. These varieties are smooth Calabi--Yau fivefolds \cite[Lemma 2.8]{mypaper_roofbundles}. In the following, we will refer to this construction as Calabi--Yau pairs (or varieties) of type $D_5$.
\begin{remark}
    As pointed out in \cite{manivel}, zero loci of general sections of $\Uc_+(2, 0)$ in $\spinor_+$ are deformation-equivalent to double-spinor Calabi--Yau fivefolds, which are intersections of the spinor tenfold with its image under a general automorphism of its ambient $\PP^{15}$. The same, of course, holds for zero loci of general sections of $\Uc_-(0, 2)$ in $\spinor_-$. In particular, the family of Calabi--Yau varieties of type $D_5$ is not locally complete, and it describes a divisor in the (locally complete) family of double-spinor Calabi--Yau fivefolds. Calabi--Yau varieties of type $D_5$ are zero loci of the normal bundle of $\SS_+$ in $\PP^{15}$, and hence they can be understood as ``intersections of infinitesimal translates''. This phenomenon also occurs in the case of intersections of general translates of $G(2, 5)$ in its Pl\"ucker embedding \cite{ottemrennemo, borisovcaldararuperry}: elements of a divisor in such family of Calabi--Yau threefolds can be described as zero loci of the normal bundle of $G(2, 5)$ in $\PP^9$.
\end{remark}

\subsection{Properties of Calabi-Yau varieties of type $D_5$}
In this subsection, we will write $\spinor$ for any of the two spinor varieties $\spinor_{\pm}$, and similarly for the tautological bundles $\shU_{\pm}$, half-spinor representations $\Delta_{\pm}$ and CY's $Y_{\pm}$.
\begin{lemma}\label{lem:Y_determines_the_section}
    Consider $\sigma_1, \sigma_2 \in H^0(\SS, \Uc(2))$ general. Then, if $Z(\sigma) = Z(\wt\sigma)$, one has $\sigma_1 = \sigma_2$ up to scalar multiplication.
\end{lemma}
\begin{proof}
    First, let us prove that there exists a map $f$ as in the diagram above, such that the square is commutative:
    \begin{equation*}
        \begin{tikzcd}
            \cdots\ar{r} & \Uc^\vee (-2)\ar[two heads]{r}{\sigma}\ar[swap, dashed]{d}{f} & \Ic_{Y|\spinor}\ar[equals]{d} \\
            \cdots\ar{r} & \Uc^\vee(-2)\ar[two heads]{r}{\wt\sigma} & \Ic_{Y|\spinor}.
        \end{tikzcd}
    \end{equation*}
    Proving that such $f$ exists is equivalent to showing that the following map is surjective:
    \begin{equation*}
        \alpha: \Hom(\Uc^\vee(-2), \Uc^\vee(-2)) \arw \Hom(\Uc^\vee(-2), \Ic_{Y|\spinor}).
    \end{equation*}
    Since $\alpha$ can be obtained by applying the functor $\Hom(\Uc^\vee(-2), -)$ to the Koszul sequence of $\wt\sigma$, the surjectivity of $\alpha$ can be checked by ensuring that 
    $$\Ext^{k-1}(\Uc^\vee(-2), \wedge^k (\shU^{\vee}(-2)))=H^{k-1}(\wedge^k (\shU^{\vee}(-2))\otimes \Uc(2))=0$$ for all $2\leq k\leq 5$. Explicitly for $k = 2, 3$ one has 
    \begin{align*}
        \wedge^k (\shU^{\vee}(-2))\otimes \Uc(2)&=\wedge^k\shU^{\vee}\otimes \wedge^{4}\Uc^{\vee}(-2k)\\
        &=\Ec_{\omega_k+\omega_4+(1-2k)\omega_5}\oplus\Ec_{\omega_{k-1}+2(1-k)\omega_5}.
    \end{align*}
 while for $k = 4$ one has $\wedge^4 (\Uc^\vee(-2))\otimes \Uc(2)\simeq \Ec_{2\omega_4-6\omega_5}\oplus \Ec_{\omega_3-6\omega_5}$, and for $k = 5$ we have $\wedge^5 (\Uc^\vee(-2))\otimes \Uc(2)\simeq \wedge^4\Uc^\vee(-8)\simeq \Ec_{\omega_4-7\omega_5}$.
 The relevant vanishings follow from the Borel--Weil--Bott theorem. By surjectivity of $\alpha$ and the fact that $\Uc$ is exceptional (and therefore simple), we deduce that $\sigma \simeq \lambda\wt \sigma$ for $\lambda\in\CC^*$.
\end{proof}
\begin{lemma}\label{lem:stability}
    The bundle $\Uc^\vee|_Y$ is slope-stable.
\end{lemma}
\begin{proof}
    We just need to apply Hoppe's criterion \cite[Proposition 1]{jardimmenetprataearp}; that is, it suffices to prove that $\wedge^k\Uc_-^\vee|_Y(0, -1)$ has no sections for $1\leq k\leq 4$. The result follows by applying the Borel--Weil--Bott theorem on the irreducible summands of bundles of the form $\wedge^l\Uc_-^\vee\otimes\wedge^k\Uc_-^\vee(-2l-1)$, for $0\leq l\leq 5$.
\end{proof}
\begin{lemma}\label{lem:Y_determines_U}
Consider the image $\spinor^g$ of $\spinor$ under $g\in\Aut(\PP(\Delta))$ and so is $\Uc^g$. Suppose that $Y$ is the zero locus of a section of $\Uc^g( 2)$. Then $\shU^g|_Y\cong\shU|_Y$.    
\end{lemma}
\begin{proof}
    Consider the two normal bundle sequences for $Y\subset \spinor\subset \PP(\Delta)$ and $Y\subset \spinor^g\subset \PP(\Delta)$:
    $$0\arw N_{Y/\spinor}\arw N_{Y/\PP(\Delta)}\arw N_{\spinor/\PP(\Delta)}|_Y\arw 0,$$
    $$0\arw N_{Y/\spinor^g}\arw N_{Y/\PP(\Delta)}\arw N_{\spinor^g/\PP(\Delta)}|_Y\arw 0,$$
    which gives a morphism $\phi: \shU(2)\cong N_{Y/\spinor}\arw N_{\spinor^g/\PP(\Delta)}|_Y$.
    By stability of $\shU$ (Lemma \ref{lem:stability}), $\phi$ is either trivial or injective. 
    When $\phi$ is trivial, the normal bundle sequences induce $N_{\spinor/\PP(\Delta)}|_Y\cong N_{\spinor^g/\PP(\Delta)}|_Y$, which is just $\shU^g(2)|_Y\cong \shU(2)|_Y$ for the case $n=5$.
    When $\phi$ is injective,  it also gives $\shU^g(2)|_Y\cong \shU(2)|_Y$ for the case $n=5$.
\end{proof}
\begin{lemma}\label{lem:Y_determines_the_spinor}
    The isomorphism class of $\Uc|_Y$ determines the embedding $\iota:Y\arw\spinor$.
\end{lemma}
\begin{proof}
    The proof is a simple adaptation of \cite[Proposition 3.6]{manivel}. In fact, we only need to ensure that \cite[Proposition 3.6, Step 3]{manivel} holds in our setting. The only difference is the Koszul resolution we use to restrict the bundles:
    \begin{equation*}
        \begin{split}
            0\arw \Oc(-8)\arw \wedge^4\Uc^\vee(-8)\arw\wedge^3\Uc^\vee(-6)\arw \\
            \arw\wedge^2\Uc^\vee(-4)\arw\Uc^\vee(-2)\arw\Oc_{\PP^{15}}\arw\Oc_{Y}\arw 0,
        \end{split} 
    \end{equation*}
    which leads to different computations. Such computations can be easily carried out with the Borel--Weil--Bott theorem, as usual.
\end{proof}
%The following corollary is identical to \cite[Corollary 7.2.5]{my_thesis}, and hence we will omit the proof.
\begin{lemma}\label{lem:contained_in_a_unique_S}
    Call $\SS^g$ the image of $\SS$ with respect of $g\in\Aut(\PP(\Delta))$ as above. Let $Y\subset \spinor$ be the zero locus of a general section of $\Uc(2)$. Then, if $Y\subset\SS^g$ as the zero locus of a general section of $\Uc^g(2)$, one has $\SS = \SS^g$, i.e. $g\in\Aut(\SS)$.
\end{lemma}

\begin{proof}
    This is a direct corollary of the combination of Lemma \ref{lem:Y_determines_U} and Lemma \ref{lem:Y_determines_the_spinor}.
\end{proof}

\begin{lemma}\label{lem:proj_bundle}
    Let $E$ be a vector bundle over any smooth projective variety $X$. Let $f:X\arw X'$ be an isomorphism, asnd $E'$ a vector bundle over $X'$ such that $f_*E = E'$. Assume there is a morphism $\tau_f:\PP(E^\vee)\arw\PP(E'^\vee)$ such that the following diagram commutes:
    \begin{equation*}
        \begin{tikzcd}[row sep = large]
            \PP(E^\vee)\ar{r}{\tau_f}\ar{d}{\pi} & \PP(E'^\vee)\ar{d}{\pi'} \\
            X\ar{r}{f} & X'.
        \end{tikzcd}
    \end{equation*}
    Then one has the following commutative diagram at the level of sections:
    \begin{equation*}
        \begin{tikzcd}[row sep = huge]
            H^0(\PP(E^\vee), \Oc_\pi(1)) \ar{r}{\tau_{f*}}\ar{d}{\pi_*} & H^0(\PP(E'^\vee), \Oc_{\pi'}(1))\ar{d}{\pi'_*} \\
            H^0(X, E) \ar{r}{f_*} & H^0(X', E').
        \end{tikzcd}
    \end{equation*}
\end{lemma}
\begin{proof}
    For any section $s$ of $E$, one has a unique section $\xi_s$ of $\Oc_\pi(1)$ (recall that $\pi_*$ is an isomorphism at the level of sections). Such $\xi_s$ acts as $\xi_s(x, v) = x, v, v\cdot s(x)$ where $x$ and $v$ are respectively coordinate on the base and on the fiber of $\PP(E^\vee)$, and ``$\cdot$'' is a fiberwise pairing between $E$ and $E^\vee$. Then, one clearly has $f_*\pi_*\xi_s = f_*s =: s'$. On the other hand, given any point $(f(x), v')\in\PP(E'^\vee)$, we have $\tau_{f*}\xi_s(f(x), v') = \xi_s \circ \tau_f^{-1}(f(x), v') = f(x), v', v\cdot s(x)$. Therefore, $\tau_{f*}\xi_{s} = s'$.
\end{proof}
\begin{lemma}
    Consider a general section $s$ of $\Oc(1,1)$ and the pair of Calabi--Yau fivefolds $Y_\pm = Z(p_{\pm *}s)\subset \SS_\pm$. Assume there is a projective isomorphism  $f: \PP(\Delta_+)\arw \PP(\Delta_-)$ such that the restriction $f|_{Y_+}: Y_+\arw Y_-$ is also an isomorphism. Then $f$ maps $\SS_+$ isomorphically onto $\SS_-$.
\end{lemma}
\begin{proof}
    By hypothesis one has $Y_-\subset \SS_-\cap f(\SS_+)$. Since one has $f(\SS_+)\simeq \SS_-$, we conclude by Lemma \ref{lem:contained_in_a_unique_S}.
\end{proof}

\subsection{Isomorphisms of orthogonal Grassmannians via Clifford multiplications}
Choose any $A\in V$ with $q(A)=-1$. This $A$ is an invertible element in the Clifford algebra $Cl(V, q)$ satisfying $A=A^{-1}$. Then the Clifford multiplication by $A$ on spinors induces a  (non-canonical) isomorphism $\Delta_+\cong \Delta_-$. Since $A^{-1}=A$, the Clifford multiplication by $A$ also induces the inverse isomorphism $\Delta_-\cong \Delta_+$. Hence, $A$ gives isomorphisms on the spinor varieties simultaneously:  $\SS_+\arw \SS_-$ and $\SS_-\arw \SS_+$. 

Consider a composition $\varphi:\SS_+\arw\SS_+$ as follows:
\begin{equation*}
    \begin{tikzcd}
        \SS_+ \ar[bend right = 30]{rr}{\varphi} \ar{r}{f} & \SS_-\ar{r}{A^{-1}} & \SS_+
    \end{tikzcd}
\end{equation*}
where $f$ is any isomorphism of $\SS_+$ with $\SS_-$. Then, we can associate to $\varphi\in\Aut(\SS_+)$ an element $M_\varphi\in\operatorname{Spin}(V_{10})$. Let $[x]\in \SS_+\subset \PP(\Delta_+)$, then $\varphi([x])=[M_{\varphi}\cdot x]$,  where $"\cdot"$ denotes the Clifford multiplication. 

Note that both $\SS_-$ and $OG(4, V_{10})$ are $\operatorname{Spin}(V_{10})$-homogeneous, and hence $M_{\varphi}\cdot(\text{-})$ induces isomorphisms on $\SS_-$ and $OG(4, V_{10})$ respectively. We may use $\varphi$ to denote the isomorphisms still. More explicitly, let $[y]\in \SS_-\subset \PP(\Delta_-)$, then $\varphi([y])=[M_{\varphi}\cdot y]$. 

The orthogonal Grassmannian $OG(4, V_{10})$ can be identified as the incidence variety of $\spinor_+\times \spinor_-$ by requiring each pair of maximal isotropic subspaces $(V_x, V_y)$ with a codimension equal to one intersection.  Then the automorphism $\varphi: OG(4, V_{10})\arw OG(4, V_{10})$ is given by
$$ ([x], [y])\mapsto ([M_{\varphi}\cdot x], [M_{\varphi}\cdot y]).$$
Note that the Clifford multiplication by $A$ induces an involution $\iota$ on $OG(4, V_{10})$:
$$\iota([x], [y])=([A\cdot y], [A\cdot x]).$$

Let us now denote by $\tau_f$ the restriction to $OG(4, V_{10})\subset \SS_+\times\SS_-$ of the following composition (which we call by the same name):
\begin{equation*}
    \begin{tikzcd}[row sep = small]
        \SS_+\times\SS_- \ar[bend left = 30]{rr}{\tau_f} \ar{r}{\varphi} & \SS_+\times\SS_- \ar{r}{\iota} & \SS_+\times\SS_- \\
        ([x], [y]) \ar[maps to]{r} & ([M_{\varphi}\cdot x], [M_{\varphi}\cdot y]) \ar[maps to]{r} & ([AM_{\varphi}\cdot y], [A M_{\varphi}\cdot x])=([AM_{\varphi}\cdot y], f([x]))
    \end{tikzcd}
\end{equation*}
Once we observe that $OG(4, V_{10})$ is the projectivization  of $\Uc_\pm(2)$ on $\SS_\pm$, and by Lemma \ref{lem:proj_bundle} we deduce that $\tau_f$ preserves the hyperplane section $s\in H^0(OG(4, V_{10}), \shO(1,1))$ defining $Y_\pm$, and we obtain that $$\tau_{f*}(s)= s,$$
up to a scaling. \\
%Let us now describe these maps more explicitly. First, we recall that one has an embedding $\operatorname{Spin}(V_{10})\subset\operatorname{Cl}(V_{10})$, where the right hand side can be isomorphically identified with $\End(\Delta_+\oplus\Delta_-)$. Through such identification, $A:\Delta_+\arw\Delta_-$ lifts to an element:
% \begin{equation*}
% \left(
% \begin{array}{cc}
%     0 & A \\
%     A^{-1} & 0
% \end{array}
% \right)
% \in\End(\Delta_+\oplus\Delta_-),
% \end{equation*}
% In this way, any $f$ can be written as $AM$ for some $M\in\operatorname{Spin}(V_{10})$, and therefore, one has
% \begin{equation*}
%     \tau_f: (x, y) \longmapsto (MA^{-1}y, AMx)
% \end{equation*}
%and to $\varphi$ we can associate a linear operator $M_\varphi\in \Aut(\Delta_+)$.
\subsection{An explicit automorphism of the space of sections}
In the following, let us call $\PP_\pm$ the projective space $\PP(\Delta_{\pm})$ for short. Suppose that there is an isomorphism $f:\spinor^+\arw \spinor^-$.
In last subsection, we construct the following morphism:
\begin{equation}
    \begin{tikzcd}[column sep = huge, row sep = tiny, /tikz/column 1/.append style={anchor=base east} ,/tikz/column 3/.append style={anchor=base west}]
        \PP_+\times\PP_-\ar{r}{\tau_f} & \PP_+\times\PP_-\\
        ([x], [y])\ar[maps to]{r} & ([AM_{\varphi}\cdot y], [AM_{\varphi}\cdot x]),
    \end{tikzcd}
\end{equation}
which fixes the subvariety $OG(4, V_{10})\subset \PP_+\times\PP_-$. Moreover, $\tau_f$ lifts to an automorphism $\wt\tau_f$ of $H^0(\PP_+\times\PP_-, \Oc(1, 1))$ by setting $\wt\tau_f(s) := s\circ\tau_f$.

Let us now describe how $\wt\tau_f$ acts explicitly on sections. A section $s\subset H^0(\PP_+\times\PP_-, \Oc(1, 1))$ can be described by a matrix $S $ in the following way:
\begin{equation}
    \begin{tikzcd}[row sep = tiny, column sep = huge, /tikz/column 1/.append style={anchor=base east} ,/tikz/column 3/.append style={anchor=base west}]
        \PP_+\times\PP_- \ar{r}{s} & \Oc(1, 1) \\
        \displaystyle{[x], [y]} \ar[maps to]{r} & \displaystyle{[x, y, x^T S  y],}
    \end{tikzcd}
\end{equation}
where, as usual, one has $[x\lambda, y\lambda', \lambda\lambda'y^T S  x] = [x, y, x^T S  y]$.

Note that $H^0(OG(4, V_{10}), \Oc(1, 1))^{\vee}\cong V_{\omega_{4}+\omega_5}\subset \Delta_+\otimes \Delta_-\cong H^0(\PP_+\times \PP_-, \shO(1,1))^{\vee}$ is the highest weight subrepresentation. Consider a section  $s\in H^0(OG(n-1, V_{2n}), \Oc(1, 1))\subset H^0(\PP_+\times \PP_-, \shO(1,1))$, viewed as a bilinear form $S\in \Delta_-^{\vee}\otimes \Delta_+^{\vee}$. Then we calculate the matrix representative of $\tau_{f*}(s)$. By definition:
    \begin{equation*}
        \begin{split}
            \wt\tau_f(s)([x], [y]) & = s\circ\tau_f([x], [y])\\
                                   & = [x,y, (AM_{\varphi}y)^T S  AM_{\varphi}x].
        \end{split}
    \end{equation*}
    Since   
    \begin{equation}
        (AM_{\varphi}y)^T S  AM_{\varphi}x=(AM_{\varphi}x)^TS^T AM_{\varphi}y=x^TM_{\varphi}^TA^TS^T AM_{\varphi}y,
    \end{equation}
     we conclude that the matrix representative of $\tau_{f*}(s)$ is
     \begin{equation}
         M_{\varphi}^TA^TS^TAM_{\varphi}.
     \end{equation}
     By Lemma \ref{lem:proj_bundle}, we find that $s=\tau_{f*}s$ up a scaling, 
     which is equivalent to
     \begin{equation}\label{keyeq}
         SA=M_{\varphi}^{T}\cdot (SA)^T\cdot (AM_{\varphi}A).
     \end{equation} 
     Note that $\bar{S}:=SA\in \Delta_+^{\vee}\otimes \Delta_+^{\vee}$ and both $M_{\varphi}$ and  $AM_{\varphi}A \in \operatorname{Spin}(V_{10})$ acts on $\bar{S}$ is induced from the Clifford multiplication on $\Delta_+$. Therefore, we obtain that
\begin{proposition}\label{prop:transposition}
    Choose any degree 1 invertible element $A$ in the Clifford algebra $Cl(V_{10})$ satisfying $A^2=1$. Let $M\in \operatorname{Spin}(V_{10})$ and $M':=AMA\in \operatorname{Spin}(V_{10})$. Consider a section $s\in H^0(OG(4, V_{10}), \Oc(1, 1))$, which represents a matrix $\bar{S}$ by composing $A$. Then 
    Then $\bar{S}= M^T \bar{S} ^T M'$.
\end{proposition}

Finally, we can show the following. 
\begin{theorem}\label{thm:not_isomoprhic_CY5}
    Let $s\in H^0(OG(4, V_{10}), \Oc(1, 1))$ be general. Then the associated Calabi--Yau fivefolds $(Y_+, Y_-)$ are not birationally equivalent.
\end{theorem}
\begin{proof}
Note that general section $s$ corresponds to a general matrix $\bar{S}\in End(\Delta)$. We may assume that $\bar{S}\in GL(\Delta)$.  
    Now suppose that $Y_+$ and $Y_-$ are birationally equivalent. Then the above discussion implies there are no matrices $S\in GL(\Delta)$ and $M\in \operatorname{Spin}(V_{10})$ satisfying the equation $S = M^{T} S^T M'$ up to a scaling. We may apply Manivel's argument in \cite[Proof of Proposition 4.4, "second case"]{manivel}. Let $G=PGL(\Delta), H=PSO(V_{10}), \Delta$ is the half-spin representation of $\operatorname{Spin}(V_{10})$. Consider the  $H\times H^{op}$-action on $G$ defined by $(M, M').S=M^TSM'$. Then $S = M^{T} S^TM'$ implies that $S^T$ and $S$ are in the same $H\times H^{op}$-orbit. The standard representation argument in \cite[Proof of Proposition 4.4]{manivel} implies that $H\subset G$ is a spherical subgroup, and hence the action of $H$ on $G/B$ is quasi-homogeneous. However, 
    \begin{equation*}
        \dim\,H= 45 <\dim\, G/B=\dim\, (U(16)/T^{16})/2=(16^2-16)/2=120.
    \end{equation*}
\end{proof}

\subsubsection{Proof of Theorem \ref{main_theorem_pairs} -- non-relative case}
Consider a pair of Calabi--Yau fivefolds $(Y_+, Y_-)$ defined by a $(1,1)$ section of $OG(4, V_{10})$ as in the previous subsection. Call $M\in OG(4, V_{10})$ the hypersurface given by the vanishing of $s$. Then, by the Cayley trick, $M$ comes with a semiorthogonal decomposition:
\begin{equation}\label{sod_from_cayley_trick}
    \begin{split}
        \dbcoh(M) =\langle  j_{+*}\bar\pi_{+}^*D(Y_{+}), &j_+^*p_{-}^*\dbcoh(\spinor_+), j_+^*p_{-}^*\dbcoh(\spinor_+)(0, 1), \\
        &j_+^*p_{-}^*\dbcoh(\spinor_+)(0, 2), j_+^*p_{-}^*\dbcoh(\spinor_+)(0, 3)\rangle,
    \end{split}
\end{equation}
and a specular one containing $j_{-*}\bar\pi_{-}^*D(Y_{-})$.
As for the problem of the $D_5$ flop, we can prove that $Y_+$ and $Y_-$ are derived equivalent by plugging in $j_{\pm*}\bar\pi_{\pm}^*D(Y_{\pm})$ an appropriate full exceptional collection, and prove that the semiorthogonal complements of $\dbcoh(Y_+)$ and $\dbcoh(Y_-)$ are derived equivalent. In particular, if we use the collection of Proposition \ref{mutated_KP_cokllection}, we obtain the same chessboard as Figure \ref{fig:starting_point} (with the difference that each object is a pullback to the divisor $M$ of a homogeneous vector bundle, instead of a pushforward from a smooth divisor to a smooth ambient variety). The following lemma shows that we can apply the same sequence of mutations of Section \ref{sec:the_mutations} verbatim:
\begin{lemma}\label{we_can_use_koszul_also_for_M}(cf. Lemma \ref{we_can_use_koszul})
    For any $F, G\in\dbcoh(OG(4, V_{10}))$, one can resolve $ \Ext_M^m(j_+^*F, j_+^*G)$ by means of the following exact sequence:
    \begin{equation*}
        \cdots\arw H^m(E, F^\vee\otimes G(-1, -1))\arw H^m(E, F^\vee\otimes G)\arw \cdots\arw  \Ext_M^m(j_+^*F, j_+^*G)\arw \cdots
    \end{equation*}
\end{lemma}
\begin{proof}
    By adjunction, one has:
    \begin{equation*}
        \Ext_M^m(j_+^*F, j_+^*G) \simeq \Ext_{OG(4, V_{10})}^m(F, j_{+*}j_+^*G).
    \end{equation*}
    To obtain the required long exact sequence, we simply need to apply $\Ext_{OG(4, V_{10})}^\bullet(F, -)$ to the tensor product of $G$ with the Koszul resolution of $j_{+*}j_+^*\Oc$.
\end{proof} 
The following corollary is obtained by applying the same mutations in Section \ref{sec:the_mutations}. 
\begin{corollary}
    The Calabi--Yau fivefolds $Y_+$ and $Y_-$ are derived equivalent.
\end{corollary}
The proof of Theorem \ref{main_theorem_pairs} is concluded when we recall that, by Theorem \ref{thm:not_isomoprhic_CY5}, $Y_+$ and $Y_-$ are derived equivalent, but not birational.
\section{Derived equivalent pairs of fibrations}
As for the flop of type $D_5$, there is a natural, relative formulation of the construction of pairs of Calabi--Yau fivefolds of type $D_5$. As in Section \ref{sec:relative_flops}, take a principal $Spin(10)$-bundle $\Uc_4$ on $B$, and consider the orthogonal Grassmann bundle $\Pc\times^G G/P_{4,5}$ with its projective bundle structures $p_{\pm}$ over $\Pc\times^G \spinor_{\pm}$. Let us call $H_\pm$ the relative hyperplane classes of these Grassmann bundles. We can choose suitable twists from $B$ such that $\Oc(1, 1)$ is a globally generated line bundle on $\Pc\times^G G/P_{4,5}$, and the restriction map of global sections of $\Oc(1, 1)$ to a fiber is surjective. Then, the zero loci $Z_\pm:= Z(p_{\pm *}s)\subset\Pc\times^G \spinor_{\pm}$ are fibrations over $B$ with general fiber isomorphic to a Calabi--Yau fivefold of type $D_5$ \cite[Lemma 2.14]{mypaper_roofbundles}.
\subsubsection{Proof of Theorem \ref{main_theorem_pairs} -- relative case}
The sequence of mutations described in Section \ref{sec:the_mutations} can be applied to the relative case, and this follows basically from the same argument we presented in Section \ref{sec:proof_theorem_flops_relative_case}. In fact, each of the mutations is of the form $\LL_{j^*V_i} \ j^*V_j \simeq \ j^*V_k$ or $\RR_{\ j^*V_i} \ j^*V_j \simeq \ j^*V_k$. By applying $j^*$ to the triangle \ref{mutation_triangle}, together with Lemma \ref{we_can_use_koszul_also_for_M} and Lemma \ref{lem:mutations}, we see that for all the mutations we consider, the isomorphism $\LL_{\ j^*\mathfrak F(V_i)} \ j^*\mathfrak F(V_j) \simeq \ j^*\LL_{\mathfrak F(V_i)}\mathfrak F(V_j)$ holds, and the right-hand side is isomorphic to $\mathfrak F(\LL_{V_i} V_j)$ by Lemma \ref{lem:mutations}. Hence, we proceed with the sequence of mutations described in Section \ref{sec:the_mutations}.

\bibliographystyle{alpha}
\bibliography{bibliography}

\end{document}